\documentclass{amsart}
\usepackage{amssymb,mathrsfs}
\usepackage{color,umoline}
\usepackage[dvipsnames]{xcolor}
\usepackage{graphicx}
\usepackage{enumitem}
\usepackage[font=small,labelfont=bf]{caption}
\usepackage{tikz, caption}
\usetikzlibrary{hobby,shapes.misc} 
\usepackage{relsize}
\usepackage[mathscr]{eucal}
\usepackage{dsfont}
\usepackage{verbatim}
\usepackage[draft]{todonotes}
\usepackage[
 citecolor=black,  
  hypertexnames=false]{hyperref}
\usepackage[dvipsnames]{xcolor}
\usetikzlibrary{arrows.meta}
\usepackage{subfig}
\usepackage{kotex}

\DeclareFontFamily{U}{mathx}{}
\DeclareFontShape{U}{mathx}{m}{n}{<-> mathx10}{}
\DeclareSymbolFont{mathx}{U}{mathx}{m}{n}
\DeclareMathAccent{\widehat}{0}{mathx}{"70}
\DeclareMathAccent{\widecheck}{0}{mathx}{"71}

\DeclareMathOperator{\Ima}{Im}

\def\wt#1{\widetilde{#1}}

\newcommand{\C}{\mathbb{C}}
\newcommand{\N}{\mathbb{N}}
\newcommand{\R}{\mathbb{R}}
\newcommand{\Z}{\mathbb{Z}}

\newcommand{\supp}{\operatorname{supp}}
\newcommand{\dist}{\operatorname{dist}}

\newtheorem{thm}{Theorem}[section]
\newtheorem{defn}[thm]{Definition}
\newtheorem{prop}[thm]{Proposition}
\newtheorem{cor}[thm]{Corollary}
\newtheorem{lem}[thm]{Lemma}
\numberwithin{equation}{section}
\newtheorem{rem}{Remark}

\newcommand{\inp}[2]{\langle #1, #2\rangle}

\newcommand{\Be}{\begin{equation}}
\newcommand{\Ee}{\end{equation}}

\makeatletter
\newcommand{\newparallel}{\mathrel{\mathpalette\new@parallel\relax}}
\newcommand{\new@parallel}[2]{%
  \begingroup
  \sbox\z@{$#1T$}
  \resizebox{!}{\ht\z@}{\raisebox{\depth}{$\m@th#1/\mkern-5mu/$}}%
  \endgroup
}
\makeatother

\definecolor{jrcol}{rgb}{0,0,1.} 
\definecolor{swcol}{rgb}{1.,0,0} 

\newcommand{\eps}{\varepsilon}

\begin{document}

\title[Restriction estimate for quasimodes]{Semiclassical $L^p$ quasimode restriction estimates in two dimensions}

\author[Oh]{Sewook Oh}
\address[Oh]{June E Huh Center for Mathematical Challenges, Korea Institute for Advanced Study, Seoul
02455, Republic of Korea}
\email{sewookoh@kias.re.kr}

\author[Ryu]{Jaehyeon Ryu}
\address[Ryu]{School of Mathematics, Korea Institute for Advanced Study, Seoul
02455, Republic of Korea} 
\email{jhryu@kias.re.kr}

\thanks{This work was supported by KIAS Individual Grant SP089101 (S. Oh) and MG087001 (J. Ryu).}

\begin{abstract}
    We establish the $L^p$ restriction estimates for quasimodes on a smooth curve in two dimensions. Our estimates are sharp for all smooth curves. As an application, we address $L^p$ eigenfunction restriction estimates for Laplace-Beltrami eigenfunctions on $2$-dimensional compact Riemannian manifolds without boundary and Hermite functions on $\R^2$. Our method involves a geometric analysis of the contact order between the curve and the bicharacteristic flow of the semiclassical pseudodifferential operator.
\end{abstract}

\maketitle

\section{Introduction}\label{sec:intro}

Let $M$ be a smooth Riemannian manifold without boundary and $P(h)$ be a semiclassical pseudodifferential operator on $M$ depending on the small parameter $h\in (0,h_0]$. Let $p(x,\xi)$ be a principal symbol of $P(h)$. Here, we use the Weyl quantization.
We call a family of functions $f=f(h)$ an $O(h)$ quasimode if $f(h)$ is a $L^2$ normalized function such that
\begin{align}\label{p:quasimode}
    P(h)f = O_{L^2}(h),\quad h\in (0,h_0].
\end{align}
In this article, we are concerned with measuring how highly $f$ can concentrate on a submanifold in $M$ as $h\to 0$. Specifically, we study the problem of obtaining $L^p$ bounds on $f$ restricted to the submanifold $H$.

This kind of research naturally has a close connection to studying the concentration of eigenfunctions. One way to approach this issue is by investigating semiclassical measures. There are extensive references regarding this line of research; see the works by Anantharaman \cite{Ana08}, Anantharaman-Koch-Nonnenmacher \cite{AKN06}, Anantharaman-Nonnenmacher \cite{AN07}, Colin De Verdiere \cite{Col85}, G\'erard-Leichtnam \cite{GL93}, Helffer-Martinez-Robert \cite{HMR87}, Zelditch \cite{Zel87}, Zelditch-Zworski \cite{ZZ96}. Another approach is to investigate the growth of $L^p$ norms of the eigenfunctions. This subject has been intensively studied since the pioneering work by Sogge \cite{Sog88}. We refer the reader to B\'erard \cite{Ber77}, Blair-Sogge \cite{BS18, BS19}, Canzani-Galkowski \cite{CG20}, Germain \cite{Ger23}, Hassell-Tacy \cite{HT15}, Sogge \cite{Sog86, Sog93, Sog11}, Tacy \cite{Tac18, Tac19, Tac21}, and references therein.

\subsubsection*{$L^p$ quasimode restriction estimates for Laplacian} To present our main result early in the introduction, we initially examine the case in which $M$ is a $d$-dimensional compact Riemannian manifold (without boundary) and $P(h) = -h^2\Delta_g -1$, where $\Delta_g$ is the Laplace-Beltrami operator on $M$. In this setting, the $L^p$ quasimode restriction estimate is equivalent to $L^p$ restriction estimate for the spectral projection $\mathds 1_{[\lambda-1,\lambda+1]}(\sqrt{-\Delta_g})$. Indeed, $\mathds 1_{[\lambda-1,\lambda+1]}(\sqrt{-\Delta_g}) f$ is an $O(h)$ quasimode of $-h^2\Delta_g -1$ for an appropriate $f\in L^2(M)$ when identifying $h=\lambda^{-1}$. This subject has been studied by Burq-G\'erard-Tzvetkov \cite{BGT07}, Hu \cite{Hu09}. The main result by Burq-G\'erard-Tzvetkov \cite{BGT07} can be stated in terms of $\mathds 1_{[\lambda-1,\lambda+1]}(\sqrt{-\Delta_g})$ as follows\footnote{We remark that the estimate \eqref{e:bgt1} is a part of the results of \cite{BGT07} related to hypersurfaces. Indeed, the authors addressed the restriction estimates for general submanifolds of arbitrary codimensions.}: For every $\lambda\ge 1$ and smooth hypersurface $H$ of $M$,
\begin{align}\label{e:bgt1}
    \|\mathds 1_{[\lambda-1,\lambda+1]}(\sqrt{-\Delta_g})\|_{L^2(M)\to L^q(H)}\lesssim \begin{cases}
\lambda^{\frac{d-1}{4} - \frac{d-2}{2q}}, & 2\le q\le \frac{2d}{d-1}, \\
\lambda^{\frac{d-1}{2}-\frac{d-1}{q}}, & \frac{2d}{d-1} \le q\le \infty.
\end{cases}
\end{align}
Additionally, if $d=2$, and  $H$ has nonvanishing geodesic curvature,
\begin{align}\label{e:bgt2}
    \|\mathds 1_{[\lambda-1,\lambda+1]}(\sqrt{-\Delta_g})\|_{L^2(M)\to L^q(H)}\lesssim \lambda^{\frac13 - \frac{1}{3q}}, \quad 2\le q\le 4.
\end{align}

The result \eqref{e:bgt2} was later extended to higher dimensions in the work by Hu \cite[Theorem 1.4]{Hu09}.  Burq-G\'erard-Tzvetkov \cite{BGT07} also proved that \eqref{e:bgt1} is sharp for all hypersurfaces if $2d/(d-1)\le q\le \infty$, and hypersurfaces containing a piece of geodesic if $2\le q\le 2d/(d-1)$. Indeed, the estimate for high $q$ is saturated by functions concentrating at a point on the scale of $\lambda^{-1}$, so the geometry of $H$ does not come into play. However, as $q$ decreases below $2d/(d-1)$, the geometry of $H$ becomes crucial. In such a case, \eqref{e:bgt1} is saturated by functions concentrating in a small tubular neighborhood of a geodesic curve. This also elucidates why \eqref{e:bgt1} could not be sharp if $H$ is curved with respect to the geodesic curve.

The objective of this article is to provide a comprehensive understanding of the relationship between $L^p$ bounds for quasimodes restricted to the hypersurface $H$ and the geometry of $H$ in a two-dimensional setting. In this setting, $H$ becomes the smooth curve $\gamma: [a,b] \to M$ parametrized by arc length. Denote by $g(\frac{D}{dt}\dot\gamma, \frac{D}{dt}\dot\gamma)$ the geodesic curvature of $\gamma$. One can consider intermediate cases for $\gamma$ where neither $g(\frac{D}{dt}\dot\gamma, \frac{D}{dt}\dot\gamma) = 0$ nor $g(\frac{D}{dt}\dot\gamma, \frac{D}{dt}\dot\gamma) \neq 0$ everywhere; for example, $g(\frac{D}{dt}\dot\gamma, \frac{D}{dt}\dot\gamma)$ may vanish at only one point on $\gamma$.
In such a case, the result by Burq-G\'erard-Tzvetkov does not give information regarding the optimality of the $L^p$ restriction estimate.

These intermediate cases can be characterized in terms of the order of contact between the curve $\gamma$ and the geodesic curve. For a smooth function $f:I \to M$ with an open interval $I\subset \R$ and $t\in I$, we define the $k$-jet $J_t^k f$ of $f$ at $t$ to be a collection of smooth functions $g: I\to M$ such that there exists a open neighborhood $U$ of $f(t)$ satisfying $(\varphi\circ f)^{(n)}(t) = (\varphi\circ g)^{(n)}(t)$ for every smooth function $\varphi: U\to \R$ and $n\le k$. Then we define
\begin{align}\label{def:sigma1}
    \sigma = \sigma_\gamma := \sup_{t,z}\Big(\sup\big\{k\in\N: \gamma\in J_t^k z\big\}\Big)
\end{align}
where the first supremum is taken over $t\in [a,b]$ and all geodesics $z: s\mapsto z_s$ parametrized by arc length such that $z_t = \gamma(t)$. Geometrically, $\sigma$ means the maximal order of contact between $\gamma$ and the geodesic curves.

We define
\begin{equation*}
    \rho(q, \sigma) := \begin{cases}
        \frac{1}{2(2\sigma+1)}\Big(1+\sigma-\frac2q\Big), & \text{if}\  \sigma < \infty, \\
        \frac14, & \text{if}\  \sigma = \infty.
    \end{cases}
\end{equation*}
The following is our main result for the Laplacian on compact Riemannian manifolds.

\begin{thm}\label{thm:specproj}
    Let $M$ be a $2$-dimensional compact Riemannian manifold without boundary. Assume that $\gamma$ is a smooth curve in $M$. Then there exists a constant $C(\gamma, M)>0$ such that for $\lambda\ge 1$,
    \begin{align}\label{e:specproj}
        \big\|\mathds 1_{[\lambda-1,\lambda+1]}(\sqrt{-\Delta_g}) \big\|_{L^2(M)\to L^q(\gamma)}\le C(\gamma,M) \lambda^{\rho(q,\sigma)},\quad 2\le q\le 4.
    \end{align}
    The estimate \eqref{e:specproj} is sharp for all smooth curves $\gamma$ in the sense that if the power $\rho(q,\sigma)$ is replaced by $\rho<\rho(q,\sigma)$, then \eqref{e:specproj} fails.
\end{thm}

\begin{rem}
    The estimate \eqref{e:specproj} was initially studied by Hu \cite{Hu09} in a model case. In this work, Hu obtained \eqref{e:specproj} under the condition that $M$ is the $2$-dimensional flat torus $\mathbb T^2$ and $\gamma$ is a smooth curve such that $\gamma^{(\sigma+1)}(0)\neq 0$ and $\gamma^{(k)}(0) = 0$ for $2\le k\le \sigma$. To contextualize this within our framework, we note that the condition on the curve is equivalent to stating that $\gamma(t)$ intersects with the geodesic line $\gamma(0)+\gamma'(0)t$ at $t=0$ with a contact order $\sigma$. Therefore, Theorem \ref{thm:specproj} serves as a generalization of Hu's result to all compact Riemannian manifolds $M$.
\end{rem}

\subsubsection*{$L^p$ quasimode restriction estimates for general pseudodifferential operators}

Theorem \ref{thm:specproj} arises as a corollary of our subsequent main result, which addresses the $L^p$ restriction estimates in the general setting of smooth Riemannian manifolds and semiclassical pseudodifferential operators. In the works by Tacy \cite{Tac10} and Hassell-Tacy \cite{HT12}, the authors extended the results of Burq-G\'erard-Tzvetkov and Hu to estimates for quasimodes, working in the semiclassical framework. To describe their work in more detail, we introduce the localization condition.

\begin{defn}
    Let $f$ be an $O(h)$ quasimode of $P(h)$ and $K$ be a compact subset of $T^*M$.
    We call $f$ localized in phase space if and only if there exists $\chi\in C_c^\infty(M)$ such that
    \begin{align}\label{i:localized}
        \supp(\chi)\subset K^\circ,\quad f(x) = \chi^w(x,hD)f(x) + O_\mathcal S(h^\infty),
    \end{align}
    where $\chi^w(x,hD)$ denotes the Weyl quantization of $\chi$ and $K^\circ$ denotes the interior of $K$. Also, by $O_{\mathcal S}(h^\infty)$ we denote a collection of functions $g\in \mathcal S(\R^2)$ such that each seminorm of $g$ is bounded by $C_N h^N$ for every $N\in \N$.
\end{defn}

To prevent the symbol $p$ from either identically vanishing or being equal to $\xi_1$, the situation in which one can obtain the $L^p$ restriction estimate merely using Sobolev inequality (see \cite{BGT04, HT12, KTZ07}), we need to impose geometric assumptions on the symbol $p$. We first assume
\[
 (A1)\ \partial_\xi p(x,\xi)\neq 0 \text{ for every } (x,\xi)\in K \text{ such that } p(x,\xi) = 0.
\]
The assumption tells us that the sets
\[
S_{p,x}:= \{\xi: p(x,\xi) = 0\},\quad x\in \pi(K),
\]
are smooth hypersurfaces in $T_x^*M$, where $\pi: T^*M\to M$ is a projection. The next assumption makes us preclude considering the case $p(x,\xi) = \xi_1$.
\[
(A2)\ \text{For every $x$, the second fundamental form of $S_{p,x}$ is positive definite.}
\]
If $p$ satisfies both the conditions (A1) and (A2), we call the symbol $p$ admissible.

In the semiclassical context, the natural counterpart to the concept of geodesic is the bicharacteristic flow of $P(h)$. The flow is given by the solution $(x,\xi)$ to the equation
\begin{align*}
    \begin{cases}
        \dot x = \partial_\xi p(x,\xi),\\
        \dot \xi = -\partial_x p(x,\xi).
    \end{cases}
\end{align*}
We note that the bicharacteristic flow coincides with the geodesic flow when $P(h)$ is $-h^2\Delta_g-1$. 
As a natural counterpart to the nonvanishing geodesic curvature condition considered in \cite{BGT07}, Hassell-Tacy \cite{HT12} investigated the scenario where $\gamma$ is curved with respect to the bicharacteristic, which means that $\gamma$ and bicharacteristic flows have a contact of order at most $1$. The following is a $2$-dimensional version of the result by Tacy \cite{Tac10} and Hassell-Tacy \cite{HT12}.

\begin{thm}[\cite{Tac10, HT12}]\label{thm:ht}
    Let $M$ be a smooth Riemannian manifold of dimension $2$ and $\gamma:[a,b]\to M$ be a smooth curve. Let $p$ be a real-valued admissible symbol. Assume that $f$ is a quasimode localized in phase space. Then we have
    \begin{align}\label{e:tacy}
    \|f\|_{L^q(\gamma)}\lesssim \begin{cases}
h^{-\frac14}, & 2\le q\le 4, \\
h^{-\frac12+\frac1q}, & 4 \le q\le \infty.
    \end{cases}
\end{align}
Furthermore, if $\gamma$ is curved with respect to the bicharacteristic, we have the improved estimate
\begin{align}\label{e:hasselltacy}
    \|f\|_{L^q(\gamma)}\lesssim h^{-(\frac13 - \frac{1}{3q})}, \quad 2\le q\le 4.
\end{align}
\end{thm}

Similarly to the result by Burq-G\'erard-Tzvetkov, there are the intermediate cases for $\gamma$ where Theorem \ref{thm:ht} does not yield sharp estimates. 
To characterize these intermediate cases, we consider the order of contact between the curve $\gamma$ and the bicharacteristic flow. 
For $(x,\xi)\in T^*M$, let $s\mapsto \kappa_s(x,\xi): I \to T^*M$ be the bicharacteristic flow of $P(h)$ with an open interval $I\subset \R$ containing $0$ such that $\kappa_0(x,\xi) = (x,\xi)$. We denote
\[
\kappa_s(x,\xi) = (z_s(x,\xi), \zeta_s(x,\xi)).
\]
For $t\in [a,b]$, we define $\Sigma_{\gamma, p}(t)$ by the set consisting of $k\ge 0$ such that there exists a reparametrization $\gamma_0: [-c,c]\to M$ of $\gamma$ and $\xi\in T_{\gamma(t)}^*M$ such that
\[
\text{\ $(\gamma(t),\xi)\in K$, \ $p(\gamma(t),\xi)=0$, \ $\gamma_0(0) = \gamma(t)$, \ $\gamma_0\in J_0^k z(\gamma(t),\xi)$}
\]
where $z(\gamma(t),\xi)$ denotes the mapping $s\mapsto z_s(\gamma(t),\xi)$.
Then we set
\[
\sigma_{\gamma,p}(t):= \sup\big(\Sigma_{\gamma, p}(t) \cup \{-1\}\big),\quad \sigma_{\gamma,p} = \sigma := \sup_{t\in [a,b]}\sigma_{\gamma,p}(t).
\]
In the case where $M$ is a compact Riemannian manifold and $P(h) = -h^2\Delta_g-1$, the above definition of $\sigma$ coincides with the definition \eqref{def:sigma1} if $K$ is suitably chosen so that $\{(x,\xi)\in T^*M: x\in \pi(K),\  |\xi|_g^2 -1 = 0\}\subset K$.

When $\sigma \le 0$, \textit{i.e.,} either $|p(\gamma(t),\xi)|\ge 1/C$ for $(\gamma(t), \xi)\in K$ or the bicharacteristic flows transversely intersects with $\gamma$, Tacy \cite{Tac10} proved that
\begin{align}\label{e:tacylq}
    \|f\|_{L^q(\gamma)}\lesssim h^{\frac1q-\frac12},\quad 2\le q\le \infty
\end{align}
holds for quasimodes $f$ of $P(h)$ localized in phase space. Therefore, from now on we assume $\sigma\ge 1$.\footnote{This assumption is natural in the sense that for every $x\in\gamma$, there exists a unique geodesic of $M$ that meets $\gamma$ at $x$ in a tangential manner.} Note that $\sigma = 1$ corresponds to the case that $\gamma$ is curved with respect to the flows.

Then the main result of this article reads as follows.

\begin{thm}\label{thm:main}
Let $M$ be a smooth $2$-dimensional Riemannian manifold without boundary. Let $h_0>0$ be a sufficiently small number. Suppose that $p = p(x,\xi)$ is an admissible real-valued symbol on $T^*M$ and $\gamma$ is a smooth curve in $M$ such that $\sigma_{\gamma,p}\ge 1$.
Then there exists a positive constant $C_1 = C_1(\gamma,p,M)$ such that for every $O(h)$ quasimode $f$ of $P(h)$ satisfying the localization property \eqref{i:localized}, 
\begin{align}\label{e:main}
    \|f\|_{L^q(\gamma)}\le C_1 h^{-\rho(q, \sigma)},\quad 2\le q\le 4,\quad h\in (0,h_0].
\end{align}
Furthermore, the estimate is sharp for all $\gamma$ and $p$ in the sense that there exists an $O(h)$ quasimode $g(h) = g(h,\gamma,p,M)$ of $P(h)$ satisfying \eqref{i:localized} such that the estimate
\begin{align*}
    \|g(h)\|_{L^q(\gamma)}\le C h^{-\rho}
\end{align*}
does not hold for any $\rho<\rho(q,\sigma)$.
\end{thm}

\begin{rem}
    Theorem \ref{thm:specproj} follows from Theorem \ref{thm:main} by setting $h = \lambda^{-1}$, $f(h)(x) = \mathds 1_{[\lambda-1,\lambda+1]}(\sqrt{-\Delta_g}) g(x)\chi_V(x)$ for $g\in L^2(M)$, where $\chi_V\in C_c^\infty(M)$ supported in an open chart $V$ of $M$. Note that $f(h)$ are localized in phase space to $K = \{(x,\xi)\in T^*M: x\in \overline{V}, |\xi|_g\le 2\}$. The optimality of \eqref{e:specproj} can be derived from the argument used by Burq-G\'erard-Tzvetkov \cite{BGT07} for proving lower bounds on the spectral projection operator. Further details will be discussed in Section \ref{sec:optimal} below.
\end{rem}

\subsubsection*{$L^p$ eigenfunction restriction estimates}
We now provide the applications of the main result to $L^p$ estimates for eigenfunctions restricted to submanifolds.

We first consider the case of Laplace-Beltrami eigenfunctions on compact Riemannian manifolds. These eigenfunctions, denoted by $\varphi_\lambda$, satisfy the equation $-\Delta_g \varphi_\lambda = \lambda^2 \varphi_\lambda$ where $\lambda\ge 0$ denotes the eigenvalues of $\sqrt{-\Delta_g}$. The $L^p$ restriction estimate for the Laplace-Beltrami eigenfunctions has attracted considerable interest from various authors. 
For references on this subject, see \cite{BS18, BR12, Che15, CS14, EP22, FS17, Ngu22, Par23, RZ23, Rez10, SZ14, Tat98, WZ21, XZ17}.

The following is an immediate consequence of Theorem \ref{thm:specproj}, which improves upon the result of Burq-G\'erard-Tzvetkov \cite{BGT07}.

\begin{cor}\label{cor:eigencpt}
    Let $M$ be a $2$-dimensional compact Riemannian manifold without boundary and $(\varphi_\lambda)$, $\lambda\ge 0$, a family of eigenfunctions of $-\Delta_g$ associated with the eigenvalue $\lambda^2$. Assume that $\gamma$ is a smooth curve in $M$. Then there exists a constant $C(\gamma, M)>0$ such that
    \begin{align}\label{e:eigencpt}
        \|\varphi_\lambda\|_{L^q(\gamma)}\le C(\gamma, M) \lambda^{\rho(q,\sigma)}\|\varphi_\lambda\|_{L^2(M)},\quad 2\le q\le 4
    \end{align}
    where $\sigma = \sigma_{\gamma,|\xi|^2_g-1}$. Furthermore, if $M = \mathbb S^2$, then \eqref{e:eigencpt} is sharp for all smooth curves $\gamma$.
\end{cor}

\begin{rem}
    While Corollary \ref{cor:eigencpt} gives sharp bounds when the manifold is the sphere, the result may be far from being optimal if the sectional curvatures of $M$ are nonpositive, see \cite{ BR12,BGT07,Hu09}.
    We believe that if the manifold is negatively curved, then the bounds in \eqref{e:eigencpt} can be improved to $\lambda^{\rho(q,\sigma)}/(\log \lambda)^c$ for some $c>0$.
\end{rem}

We now look for another application of Theorem \ref{thm:main} to Hermite functions on $\R^2$.
Let $M = \R^2$ and consider the Hermite operator $H = -\Delta + |x|^2$.
The eigenvalues of the Hermite operator are $2n+2$, $n\in \N_0:=\N\cup\{0\}$, and the eigenfunctions are called Hermite functions.

The problem of estimating $L^p$ bounds of the Hermite functions has been studied in various settings. Koch-Tataru \cite{KT05} and Jeong-Lee-Ryu \cite{JLR23} proved global $L^p$ estimates for the Hermite functions, except for $d=2$, $p = 10/3$. Koch-Tataru-Zworski \cite{KTZ07} studied this problem for general Schr\"odinger operators in a semiclassical setting. Local $L^p$ bounds over balls or compact subsets of $\R^d$ have also drawn the attention of several authors, see Jeong-Lee-Ryu \cite{JLR24}, Thangavelu \cite{Tha98}, Wang-Zhang \cite{WZ23}.

In this article, we obtain $L^p$ bounds of restrictions of the Hermite functions to smooth curves in two dimensions. Let $p_H(x,\xi)=|x|^2+|\xi|^2-1$.

\begin{cor}\label{cor:eigensch}
Suppose that $\gamma: [a,b]\to \R^2$ is a smooth curve in $\{ |x| < 1\}$. Let $\varphi_\lambda$ be $L^2$-normalized Hermite functions associated with the eigenvalues $\lambda^2=2n+2$. Then we have
    \begin{align}\label{e:eigensch}
        \|\varphi_\lambda\|_{L^q(\lambda\gamma)} \le C(\gamma)\lambda^{-1+\frac1q+2\rho(q,\sigma)}
    \end{align}
    for $2\le q\le 4$ with a constant $C(\gamma)>0$, where $\lambda \gamma = \{\lambda\gamma(t): t\in [a,b]\}$ denotes a dilation of $\gamma$ and $\sigma = \sigma_{\gamma, p_H}$. 
    Furthermore, \eqref{e:eigensch} is sharp for all smooth curves $\gamma\subset \{|x|<1\}$.
\end{cor}
The proof of implication from Theorem \ref{thm:main} to \eqref{e:eigensch} is not complicated, so we include it in the introduction. Let $h = \lambda^{-2}$ and $f_{\lambda}(x) := \lambda \chi(x) \varphi_\lambda (\lambda x)$ where $\chi\equiv 1$ on $B(0,1)$ and $\supp(\chi)\subset B(0,2)$. Here, $B(0,r)$ denotes the ball centered at the origin with radius $r$. Then we have $p_H^w(x,hD) f_{\lambda} = O_{L^2}(h)$.
 Note also that $f_{\lambda}$ is localized in phase space to $K = \{|x|\le 2\}\times \{|\xi|\le 2\}$. Hence \eqref{e:eigensch} follows from \eqref{e:main} and scaling.
 For the proof of optimality, see Section \ref{sec:optimal}.

\begin{rem}
    Using the argument in \cite{KTZ07}, we can extend the above result to the general Schr\"odinger operator $-\Delta+V$, where $V$ is a smooth nonnegative potential on $\R^2$ such that
\begin{align*}
    |\partial_x^\alpha V(x)|\lesssim \langle x\rangle^2\ \text{for $\alpha\in \N_0^2$},\quad |V(x)|\ge c\langle x\rangle^2\ \text{for $|x|\ge R$}
\end{align*}
with constants $c,R>0$.
\end{rem}

\subsubsection*{Our strategy} The proof of Theorem \ref{thm:main} by and large comprises two crucial parts. The first part is to develop a localization scheme that reduces the problem to studying model cases stated in Proposition \ref{prop:main} below. In Section \ref{sec:redgloc}, we localize the quasimode $f$ in phase space, which gives rise to the decomposition of $\gamma$ into disjoint smaller curve segments $(\gamma_j)$ such that for each $\gamma_j$, either $\sigma_{\gamma, p}(t)\le 0$ for $\gamma(t)\in \gamma_j$ or $\sigma_{\gamma, p}(t)\ge 1$ for $\gamma(t)\in \gamma_j$.
For curves falling into the first category, we make use of the known estimate \eqref{e:tacylq} due to Tacy \cite{Tac10}. In the second category, we will show that the set of points of higher contact order $\{t:\gamma(t)\in \gamma_j,\ \sigma_{\gamma,p}(t)\ge 2\}$ contains at most finite elements, provided there is no point in $\gamma_j$ where the bicharacteristic curve intersects with $\gamma_j$ with infinite order of contact, see Lemma \ref{lem:finite} below. This observation allows us to decompose $\gamma$ ensuring $|\{t\in [a,b]:\sigma_{\gamma,p}(t)\ge 2\}|\le 1$ if $\sigma<\infty$.

After reducing to the above setting, in Section \ref{sec:redpropa}, we proceed with establishing the $L^2$ bound \eqref{e:l2u}. Instead of invoking known results in the theory of fold singularities (\cite{Com99, GS94, MT85, PS90}) as in \cite{Hu09, HT12}, we take an approach that focuses on establishing the resolvent estimate
\[
\big\|\chi^w(x,hD)(h^{-1}P+i)^{-1}\big\|_{L^2\to L^2(\gamma)}\lesssim h^{-\rho(2,\sigma)}
\]
which is equivalent to \eqref{e:l2u}. Combining a $TT^*$-argument and a known integral representation formula for the propagator $e^{ih^{-1}P(h)}$, we reduce the matter to estimating an operator norm of an oscillatory integral operator from $L^2(\gamma)$ to $L^2(\gamma)$. 

Our strategy for accomplishing this task is based on obtaining favorable bounds on the difference of phases $\Phi(u,w)-\Phi(u,v)$ (see Proposition \ref{prop:phibound} below). Estimating this phase difference is the second crucial part of our proof, which is discussed in Section \ref{sec:phase}. The main idea is to establish the relation between $\partial_u\partial_v\Phi(u,v)$ and $z_{v-u}(\gamma(u),\xi(u)) - \gamma(v)$, the difference between the curve and a particular bicharacteristic flow of $P(h)$ intersecting each other when $v=u$. Subsequently, we prove that this difference term can be approximated by a polynomial \eqref{i:diff2}.

We point out that the phase difference $\Phi(u,w)-\Phi(u,v)$ exhibits different behavior depending on whether $\sigma$ is odd or even. 
The reason for this behavior is connected to the geometry of intersection between the curve $\gamma$ and the bicharacteristic curve $v\mapsto z_{v-u}(\gamma(u),\xi(u))$. See Figure \ref{fig:casesigma}. Given this distinctive behavior, we estimate both the first and second derivatives of $\Phi(u,w)-\Phi(u,v)$ in $u$ when $\sigma$ is even. In contrast, when $\sigma$ is odd, estimating only the first derivative of $\Phi(u,w)-\Phi(u,v)$ proves sufficient for our purposes.

Our proof is notably constructive and does not rely on known results in the theory of fold singularities. An advantageous aspect of this approach is that one can establish the optimality of Theorem \ref{thm:main} in a straightforward manner. Indeed, in Section \ref{sec:optimal}, we achieve this by utilizing the argument we devise to prove the upper bound \eqref{e:main}. Additionally, in this section, we provide the proof of the optimality of Corollary \ref{cor:eigencpt} and Corollary \ref{cor:eigensch}.

\subsubsection*{Notations} Let $\delta>0$ and $K$ be a subset of $M$ (or $T^*M$). By $N_\delta K$ we denote a $\delta$-neighborhood of $K$. For given non-negative quantities $A$ and $B$, by $A\lesssim B$ we mean that there exists a constant $C>0$ such that $A\le CB.$ 
We write $A\sim B$ if $A\gtrsim B$ and $A\lesssim B$. By $D = O(A)$, we mean $|D|\lesssim A$. Moreover, we denote $A\gg B$ if there is a sufficiently large constant $C>0$ such that $A\ge CB$.

\section{Order of contact between curve and bicharacteristic flow}\label{sec:redgloc}

In this section, preliminary considerations for proving Theorem \ref{thm:main} are discussed. Most of this section is devoted to analyzing the curve $\gamma$ and the bicharacteristic flow of $P(h)$ from a geometrical perspective, specifically by examining the order of contact. This analysis allows us to reduce the problem to obtaining \eqref{e:main} only for $\gamma: [a,b]\to M$ such that $\sigma_{\gamma,p}(t)\ge1$ for all $t\in [a,b]$. 

\subsection{Basic considerations}\label{ssec:basic} From now on, we assume that $f$ is an $O(h)$ quasimode satisfying the localization condition \eqref{i:localized}.
As mentioned in Tacy \cite{Tac10}, $f$ is also a quasimode of $p^w(x,hD)$. Thus we may assume that $P(h) = p^w(x,hD)$. By the localization condition, we can also assume $p$ is supported in $N_\delta K$ for a small $\delta>0$.

We note that \eqref{e:main} easily follows once we prove it only for $q=2$. In fact, Tacy \cite{Tac10} proved the $L^4$ estimate $\|f\|_{L^4(\gamma)}\lesssim h^{-1/4}$. This coincides with \eqref{e:main} with $q=4$, thus by the H\"older's inequality, \eqref{e:main}
would follow once we verify
\begin{align}\label{e:mainl2}
    \|f\|_{L^2(\gamma)}\lesssim h^{-\rho(2,\sigma)}.
\end{align}

Decomposing the support of $\chi$ into finite small pieces, we can move from working on the whole manifold to working on a coordinate chart $V$ of $M$.  Thus we may identify $M$ with $\R^2$ and assume that $K\subset T^*\R^2\cong \R^2\times \R^2$. Choosing $K$ large, we can further assume that $K = K_P\times K_F \subset \R^2\times \R^2$. Let $\chi_P\in C_c^\infty(\R^2)$, $\chi_F\in C_c^\infty(\R^2)$ such that $\supp(\chi_P)\subset K_P^\circ$, $\supp(\chi_F)\subset K_F^\circ$, and $\chi_P(x)\chi_F(\xi) = 1$ for $(x,\xi)\in\supp(\chi)$. Since the localization condition \eqref{i:localized} continues to hold if we replace $\chi$ by $\chi_P\chi_F$, we may assume that $\chi(x,\xi) = \chi_P(x)\chi_F(\xi)$. 

Taking into account additional decompositions of $\supp(\chi_P)$ and $K_P$ if necessary, we may also assume that $\gamma\,\cap K_P$ is path-connected. Note that the contribution from $\gamma\smallsetminus(\gamma\,\cap K_P)$ is negligible. Consequently, \eqref{e:main} would follow once we estimate $\|f\|_{L^q(\gamma)}$ in the setting of $M=\R^2$, $\gamma\subset K_P$. In this setting, the smooth curve $\gamma$ is $\gamma:[a,b]\to \R^2$ and $p(x,\xi)$ is an admissible symbol on $T^*\R^2\cong \R^2\times \R^2$ supported in $N_\delta (K_P\times K_F)$, $\delta>0$. To provide a clear description of the contact of order $\sigma_{\gamma, p}(t)$ at the endpoints $t=a, b$, we use the fact that $\gamma$ has a smooth extension $\wt\gamma:[a-\delta, b+\delta]\to \R^2$ with small $\delta>0$, justified by the Whitney extension theorem. For simplicity, we identify $\wt\gamma = \gamma$ and denote by $\gamma|_{[a,b]}$ the image of $\gamma$ restricted to $[a,b]$.

\subsection{Localization argument}
We consider the set
\[
\mathcal Z_{\gamma(t)}:=\{\xi\in K_F: p(\gamma(t), \xi) = 0,\ \partial_\xi p(\gamma(t),\xi) \parallel \dot\gamma(t)\},\quad t\in [a,b],
\]
where $\partial_\xi p(\gamma(t),\xi) \parallel \dot\gamma(t)$ means that $\partial_\xi p(\gamma(t),\xi)$ is parallel to $\dot\gamma(t)$. This set encapsulates all points in frequency space such that the bicharacteristic curve $s\mapsto z_s(\gamma(t_0),\xi)$ tangentially intersects with $\gamma$.
Note that $\mathcal Z_{\gamma(t)}$ is nonempty only when $\sigma_{\gamma, p}(t)\ge 1$. The existence of such $t$ is guaranteed by the assumption $\sigma\ge 1$. The following lemma provides more detailed information about $\mathcal Z_{\gamma(t)}$.

\begin{lem}\label{lem:implicitg}
    Let $t_0\in [a,b]$ and $\sigma_{\gamma,p}(t_0)\ge 1$. Then the following hold.
\begin{enumerate}[topsep= -1pt, itemsep= 2pt]
    \item [i)] The set $\mathcal Z_{\gamma(t_0)}$ has finite elements.
    \item [ii)] Denote $\mathcal Z_{\gamma(t_0)} = (\xi_k^*)_{k=1}^N$. For each $1\le k\le N$, there is an open interval $(t_0-c, t_0+c)$ with $c>0$ such that there exists a unique smooth function $\xi_k: (t_0-c, t_0+c)\to \R^2$ such that $\xi_k(t_0) = \xi_k^*$, $p(\gamma(t), \xi_k(t)) = 0$, and $\partial_\xi p(\gamma(t),\xi_k(t)) \parallel \dot\gamma(t)$ for every $t\in (t_0-c, t_0+c)$. Moreover, the constant $c$ can be chosen to depend only on $p, \gamma, K$.
\end{enumerate}
\end{lem}

\begin{proof}
Let $\delta>0$ be a small constant. Define the function $\Theta : (t,\xi)\in [a,b]\times N_\delta K_F \to \R^2$ by
\[
\Theta(t,\xi) = (\langle\partial_\xi p(\gamma(t),\xi),R_{\pi/2}\dot\gamma(t)\rangle, p(\gamma(t), \xi)),
\]
where $R_{\pi/2}$ is the matrix for rotation by $\pi/2$. Clearly, $\Theta(t_0, \xi_0) = 0$ for $\xi_0\in\mathcal Z_{\gamma(t_0)}$. Fix $\xi_0$. Note that $\Theta$ is smooth near $(t_0,\xi_0)$.

We claim that $\partial_\xi \Theta(t_0, \xi_0)$ is invertible. Assuming this for the moment, one can easily show that $\mathcal Z_{\gamma(t_0)}$ has no limit point. 
From this, \textit{i)} follows. The second assertion \textit{ii)} follows from the implicit function theorem.

To prove the claim, we calculate
\[
\partial_\xi \Theta(t,\xi) = (\partial_\xi^2 p(\gamma(t),\xi)R_{\pi/2}\dot\gamma(t), \partial_\xi p(\gamma(t), \xi)).
\]
Combining with $R_{\pi/2}\dot\gamma(t_0)\perp \partial_\xi p(\gamma(t_0), \xi_0)$ and the condition (A2), we obtain
\begin{align}\label{clm:theta}
    \langle \partial_\xi^2 p(\gamma(t_0),\xi_0)R_{\pi/2}\dot\gamma(t_0), v \rangle \neq 0,\quad v\perp \partial_\xi p(\gamma(t_0), \xi_0),
\end{align}
from which the claim follows.
\end{proof}

Making use of Lemma \ref{lem:implicitg}, we can prove the following.

\begin{lem}\label{lem:covering}
    There exists a finite open cover $(U_k)_{k=1}^N$ of $\gamma|_{[a,b]}$ such that for each $k$, $\gamma\cap \overline U_k$ is a smooth curve $\gamma: [a_k,b_k]\to \R^2$, and one of the following holds.

\begin{enumerate}[topsep=-1pt, itemsep=2pt]
    \item [a)] $\mathcal Z_{\gamma(t)}\cap \supp(\chi_F) = \emptyset$ for every $\gamma(t)\in U_k$.
    \item [b)] There is a finite collection $(B_{k,j})_{j=1}^{N_k}$ of open balls in $K_F$ such that for every $j$, there exists a unique smooth function $\xi_{k,j}: [a_k, b_k]\to \R^2$ such that 
    \[
    \xi_{k,j}(t)\in\mathcal Z_{\gamma(t)}\cap B_{k,j}
    \]
    and 
    \[
    \bigcup_{t\in [a_k,b_k]} \big(\mathcal Z_{\gamma(t)}\cap \supp(\chi_F)\big) \subset \bigcup_{j=1}^{N_k} \Ima{(\xi_{k,j})}
    \]
    for every $1\le k\le N$, where $\Ima{(\xi_{k,j})}$ denotes the image of $\xi_{k,j}$.
\end{enumerate}
\end{lem}

\begin{proof}
Choose $t\in [a,b]$ such that $\mathcal Z_{\gamma(t)}\neq\emptyset$ and fix it. By Lemma \ref{lem:implicitg}, $\mathcal Z_{\gamma(t)} = (\zeta_{j})_{j=1}^{N}$ and there exist a constant $c=c(p,\gamma,K)>0$ and  smooth functions $\xi_{j}: (t-c,t+c)\to \R^2$ such that $\xi_{j}(t) = \zeta_{j}$, $p(\gamma(s), \xi_{j}(s)) = 0$, and $\partial_\xi p(\gamma(s), \xi_{j}(s))$ is parallel to $\dot\gamma(s)$ for $s\in (t-c,t+c)$ and $1\le j\le N$. Now, take $c>0$ small enough so that if $\Ima(\xi_j)\cap \supp(\chi_F)\neq \emptyset$, then $\Ima(\xi_j) \subset K_F$.
Then we choose a subcollection $(\xi_{j}^*)_{j=1}^{N}\subset (\xi_{j})_{j=1}^{N}$ such that $\Ima(\xi_j^*)\cap \supp(\chi_F)\neq \emptyset$.  
We claim that
    \[
    \bigcup_{s\in (t-c/2, t+c/2)} \big(\mathcal Z_{\gamma(s)}\cap \supp(\chi_F)\big) \subset \bigcup_{j=1}^{N} \Ima{(\xi_{j}^*)}.
    \]
    Let $(s_*,\zeta)\in (t-c/2, t+c/2)\times \R^2$ such that $\zeta\in \mathcal Z_{\gamma(s_*)}\cap\,\supp(\chi_F)$. By Lemma \ref{lem:implicitg}, there exists a smooth function $\xi_\zeta: (s_*-c/2,s_*+c/2)\to \R^2$ such that $\xi_\zeta(s_*) = \zeta$, $p(\gamma(u), \xi_\zeta(u)) = 0$, and $\partial_\xi p(\gamma(u), \xi_\zeta(u))$ is parallel to $\dot\gamma(u)$ for $u\in (s_*-c/2,s_*+c/2)$. 
    Since $t\in(s_*-c/2,s_*+c/2)$, $\zeta \in \Ima(\xi_\zeta)\subset \Ima(\xi_j)$ for some $j$, implying $\xi_j\in (\xi_j^*)_{j=1}^N$. Thus the claim follows.

    For $t\in [a,b]$, let $U_t$ be a small neighborhood of $\gamma(t)$ such that $\gamma\cap\, U_t\subset \{\gamma(s):s\in(t-c/2,t+c/2)\}$. Clearly, $(U_t)_{t\in[a,b]}$ is a cover of $\gamma|_{[a,b]}$.
By the choice of $c$ in the above, we have $\{\xi_j^*(s):s\in(t-c/2,t+c/2)\}\subset K_F$ for all $j$.
Therefore, there exists a collection of disjoint open balls $(B_{t,j})_{j=1}^{N_t}$ satisfying $B_{t,j}\subset K_F$, $\xi_{j}^*(s) \in B_{t,j}$ for every $s\in [t-c/2,t+c/2]$. Choosing a finite subcover $(U_k)_{k=1}^N$ of $\gamma$ completes the proof.
\end{proof}
\begin{figure}
    \centering
    \begin{tikzpicture}[use Hobby shortcut] 
\draw[dashed] (-2,0) .. (-1,2) .. (1,2) .. (2,1.5) .. 
(2.5,0) .. (0,-1.5) .. (-2,0); 
\draw (-1,-2) .. (0,-0.7) .. (0.7,1.5) .. (1.5,2.5);
\draw[densely dashed] (0.3,0.5) circle[radius=0.5cm];
\node[scale=0.9, right] at (0.75,0.4) {$U_k$};
\node[scale=1.1, right] at (1.5,2.5) {$\gamma$};
\node[scale=0.9, above] at (-0.2,2.2) {$\supp(\chi_P)$};
\end{tikzpicture}
\qquad\qquad
\begin{tikzpicture}[use Hobby shortcut] 
\draw[dashed] (-2,0) .. (-1.5,1.3) .. (-1,1.7) .. (0.5,2) .. 
(1,1.8) .. (3,0) .. (1.5,-1) .. (0,-1.5) .. (-1.5,-1) .. (-2,0); 
\draw[color=blue] (-1.6,1.75) .. (-1.4,1.6) .. (-1.2,1);
\draw[color=blue] (1.1,-0.8) .. (1.4,-1) .. (1.6,-1.3) .. (1.8, -1.5);
\draw[color=blue] (1,1.4) .. (1.1,1.7) .. (1.6,2) .. (1.7,2.2);
\draw[color=blue] (-0.7,-0.3) .. (-0.3,-0.2) .. (0,0) .. (0.1,0.2);
\draw[densely dashed] (-0.3,-0.1) circle[radius=0.6cm];
\draw[densely dashed] (1.3,1.83) circle[radius=0.65cm];
\draw[densely dashed] (1.5,-1.1) circle[radius=0.65cm];
\draw[densely dashed] (-1.4,1.4) circle[radius=0.6cm];
\node[scale=0.9, above] at (1.3,2.5) {$B_{k,j}$};
\node[scale=0.9, right] at (2.5,-0.8) {$\supp(\chi_F)$};
\end{tikzpicture}
\caption{The open neighborhood $U_k$ (left) and a collection of balls $B_{k,j}$ (right). The blue curve segments denote the set $\cup_{t\in (a_k,b_k)} \mathcal Z_{\gamma(t)}$.}
\label{fig:ukbkj}
\end{figure}
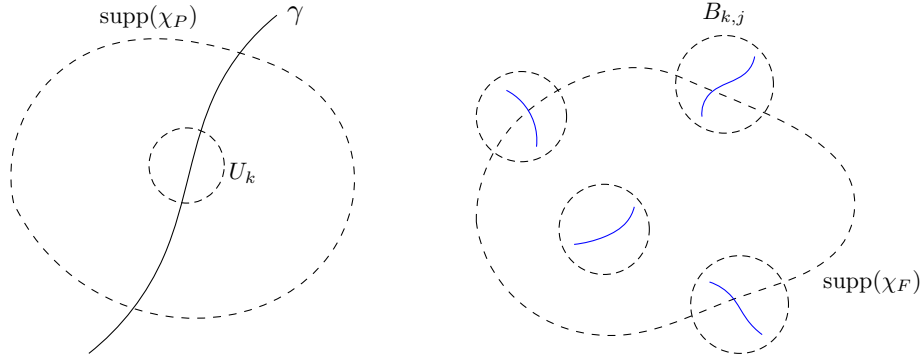

Let $(U_k)_{k=1}^N$ be the open cover of $\gamma|_{[a,b]}$ constructed in Lemma \ref{lem:covering}. Let $\mathcal I_1$, $\mathcal I_2$ be disjoint sets of indices such that $\mathcal I_1\cup \mathcal I_2 = \{1\le k\le N\}$ and the elements of $\mathcal I_1$, $\mathcal I_2$ satisfy the conditions \textit{a)}, \textit{b)} in Lemma \ref{lem:covering}, respectively. For $k\in \mathcal I_2$, let $(B_{k,j})_{j=1}^{N_k}$ be a collection of open balls in $K_F$ constructed in Lemma \ref{lem:covering}. We define the collections of sets $\mathcal U_1, \mathcal U_2, \mathcal U_3$ by
\[
\mathcal U_1 := \big(U_k\times N_\delta(K_F)\big)_{k\in \mathcal I_1},\  \mathcal U_2:= \big(U_k\times W_{k,>0}\big)_{k\in \mathcal I_2}, \  \mathcal U_3:= \big(U_k\times B_{k,j}\big)_{k\in \mathcal I_2, 1\le j\le N_k}
\]
for small $\delta>0$, where $W_{k,>0}$ is an open subset of $\R^2$ such that for $k\in \mathcal I_2$, the sets $W_{k,>0}, (B_{k,j})_{j=1}^{N_k}$ cover $K_F$ and $\mathcal Z_{\gamma(t)} \cap \supp(\chi_F) \cap W_{k,>0} = \emptyset$ for every $t\in [a_k,b_k]$.
Note that $\mathcal U_1 \cup \mathcal U_2 \cup \mathcal U_3$ is an open cover of $N_\delta\gamma\times K_F$ for small $\delta>0$. Let $\mathcal X$ be a partition of unity of $N_\delta\gamma\times K_F$ associated to the open sets in $\mathcal U_1 \cup\,\mathcal U_2 \cup\,\mathcal U_3$. Using the localization condition, we write
\[
f = \sum_{\chi\in \mathcal X} \chi^w(x,hD) f + O_{\mathcal S}(h^\infty).
\]
Let $\chi\in \mathcal X$. If $\supp(\chi)$ is contained in an open patch in $\mathcal U_1 \cup\,\mathcal U_2$, then by construction, there is no $(\gamma(t),\xi)\in \supp(\chi)$ such that $\gamma$ and $z_s(\gamma(t), \xi)$ meets tangentially at $\gamma(t)$. Using the result in Tacy \cite{Tac10}, we obtain the estimate $\|\chi^w(x,hD) f\|_{L^2(\gamma)}\lesssim 1$, which is better than what we desire. Hence we only need to estimate $\chi^w(x,hD) f$ for $\chi$ such that $\supp(\chi)\subset W$ for some $W\in \mathcal U_3$.

Consequently, from the discussion above, we conclude that it suffices to prove \eqref{e:mainl2} by further assuming that for every $t\in [a,b]$, $ \sigma_{\gamma,p}(t)\ge1$ and $\mathcal Z_{\gamma(t)}\cap\,\supp(\chi_F)$ contains at most one element.

\subsection{Order of contact}\label{ssec:contact}
By the above assumption and Lemma \ref{lem:covering}, there is a unique smooth function $\xi: [a-\delta,b+\delta]\to \R^2$ with small $\delta>0$ such that $p(\gamma(t), \xi(t)) = 0$ and $\partial_\xi p(\gamma(t), \xi(t))$ is parallel to $\dot\gamma(t)$.
By the Picard-Lindel\"of theorem, there exists a smooth function $L:[-T-\delta_1,T+\delta_2] \to [a-\delta, b+\delta]$ for some $T, \delta_1, \delta_2>0$ such that $L(-T) = a$, $L(T) = b$,
\[
\partial_t\big(\gamma(L(t))\big) = \partial_\xi p(\gamma(L(t)), \xi(L(t))) = \dot z_0(\gamma(L(t)), \xi(L(t))).
\]
For simplicity, we identify $\gamma\circ L, \xi\circ L$ with $\gamma, \xi$, respectively.
In this setting, the order of contact $\sigma_{\gamma, p}(t)$ at $\gamma(t)$ is reformulated as
\[
\sigma_{\gamma, p}(t)= \sup\{k\in \N:\ \gamma^{(k')}(t) = \partial_s^{k'}z_s(\gamma(t), \xi(t))\big|_{s=0}\ \text{for $k'\le k$}\},\quad t\in [a,b].
\]
In considering the contact order at the endpoints $t = a,b$, we use the extension $\wt \gamma$. Notice that the notation does not depend on the specific choice of $\wt\gamma$.
Furthermore, it is also independent of the choice of the chart. In fact, consider the transition map $\tau_{\alpha,\beta}: \varphi_\alpha(V_\alpha\cap V_\beta)\to \varphi_\beta(V_\alpha\cap V_\beta)$. Then using the canonical transformation $(x,\xi)\mapsto (\tau_{\alpha,\beta}(x), (\tau_{\alpha,\beta}'(x)^\intercal)^{-1}\xi)$ justifies that $\sigma_{\gamma, p}(t)$ remains unchanged when using another chart $V_\beta$ instead of $V_\alpha$ to describe it (see, for example, \cite[Section 2.3]{Sog14}).

Then, the sets $\mathcal G_k^{\gamma, p}$ for the higher order of contact can be defined as follows: for $k\ge 1$,
\begin{align*}
    \mathcal G_k^{\gamma, p} := \big\{t\in [a,b]: \sigma_{\gamma, p}(t)\ge k \big\},\quad \mathcal G_\infty^{\gamma, p} := \bigcap_{k=2}^\infty \mathcal G_k^{\gamma, p}.
\end{align*}
If $\sigma_{\gamma,p} < \infty$, then we can show the following.

\begin{lem}\label{lem:finite}
Assume that $\mathcal G_\infty^{\gamma, p} = \emptyset$. Then $\mathcal G_2^{\gamma, p}$ is finite.
\end{lem}

\begin{proof}
Let $t\in \mathcal G_2^{\gamma, p}$. First, we differentiate $p(\gamma(t), \xi(t)) = 0$ and obtain
\begin{align}\label{i:perp1}
    \inp{\partial_\xi p(\gamma(t), \xi(t))}{\dot\xi(t) + \partial_x p(\gamma(t), \xi(t))} = 0.
\end{align}
We also differentiate both sides of $\partial_\xi p(\gamma(t),\xi(t)) = \dot\gamma(t)$ to get
\begin{align}\label{i:gt2}
    \gamma^{(2)}(t) = \partial_x^\intercal\partial_\xi p(\gamma(t),\xi(t))\dot{\gamma}(t) + \partial_\xi^\intercal \partial_\xi p(\gamma(t),\xi(t))\dot{\xi}(t).
\end{align}
On the other hand, a calculation gives
\begin{align}\label{i:gt3}
    \partial_s^2 z_s(\gamma(t), \xi(t))\big|_{s=0} = \partial_x^\intercal\partial_\xi p(\gamma(t),\xi(t))\dot{\gamma}(t) - \partial_\xi^\intercal \partial_\xi p(\gamma(t),\xi(t))\partial_x p(\gamma(t),\xi(t)).
\end{align}
Subtracting \eqref{i:gt3} from \eqref{i:gt2} yields
\[
\partial_\xi^\intercal \partial_\xi p(\gamma(t), \xi(t))\big(\dot{\xi}(t) + \partial_x p(\gamma(t), \xi(t))\big) = \gamma^{(2)}(t) - \partial_s^2 z_s(\gamma(t), \xi(t))\big|_{s=0} = 0.
\]
 Suppose that $\dot{\xi}(t) + \partial_x p(\gamma(t), \xi(t))\neq 0$. However, by the condition (A2), this is a contradiction to \eqref{i:perp1}. Therefore, we conclude that
\begin{align}\label{i:relA2}
    \mathcal G_2^{\gamma, p} = \big\{t\in [a,b] : \dot\gamma(t) = \partial_\xi p(\gamma(t),\xi(t)),\quad \dot\xi(t) = -\partial_x p(\gamma(t), \xi(t))\big\}.
\end{align}

We now prove the lemma. Suppose that $\#(\mathcal G_2^{\gamma, p}) = \infty$. Then there exists a sequence $(t_k)_{k=1}^\infty\subset \mathcal G_2^{\gamma, p}$ converging to $t_\circ\in [a,b]$ as $k\to \infty$. Let us set
\[
\mathbf H(t):=  \dot\xi(t)+ \partial_x p(\gamma(t), \xi(t)),\quad t\in [a,b].
\]
By \eqref{i:relA2}, $\mathbf H(t_k) = 0$ for $k\ge 1$. Thus by Taylor's theorem, $\partial_t^n\,\mathbf H(t_\circ) = 0$ for $n\ge 0$. Indeed, suppose that there exists $N\in\N$ such that $\partial_t^n\,\mathbf H(t_\circ) = 0$ for $n<N$, but $\partial_t^N\,\mathbf H(t_\circ)\neq 0$. Then by Taylor's theorem, we have $|\mathbf H(t)|\sim |t-t_\circ|^N$ for $t$ near $t_\circ$, which is a contradiction to $\mathbf H(t_k) = 0$.

We claim that $t_\circ\in \mathcal G_\infty^{\gamma, p}$. This would lead to the desired contradiction, completing the proof. To show the claim, we verify
\begin{align}\label{i:clmgamxi}
    \gamma^{(k)}(t_\circ) = \partial_s^{k}z_s(\gamma(t_\circ), \xi(t_\circ))\big|_{s=0},\quad \xi^{(k)}(t_\circ) = \partial_s^{k}\zeta_s(\gamma(t_\circ), \xi(t_\circ))\big|_{s=0}
\end{align}
for every $k\ge 2$. We prove this by induction on $k$. First, note that the statement for $k=2$ immediately follows from \eqref{i:gt2}, \eqref{i:gt3}, and $\mathbf H(t_\circ) = 0$. Next, assume that \eqref{i:clmgamxi} is valid for $k\le n-1$ with $n\ge 3$. We write
\begin{align*}
    \gamma^{(n)}(t_\circ) &= \partial_t^{n-1}\big(\partial_\xi p(\gamma(t), \xi(t))\big)\big|_{t= t_\circ},\\
    \partial_s^{n}z_s(\gamma(t_\circ), \xi(t_\circ))\big|_{s=0} &= \partial_s^{n-1}\big(\partial_\xi p(z_s(\gamma(t_\circ),\xi(t_\circ)), \zeta_s(\gamma(t_\circ),\xi(t_\circ)))\big)\big|_{s=0}.
\end{align*}
A straightforward calculation using the chain rule and the induction hypothesis shows that the two terms on the right-hand sides in the first and second lines should be equal. Hence the first identity in \eqref{i:clmgamxi} with $k=n$ follows. The second identity for $\xi, \zeta_s$ can be proven in the same manner. One additionally needs to use $\partial_t^n \mathbf H(t_\circ)= 0$. Consequently, \eqref{i:clmgamxi} holds for every $k\ge 2$.
\end{proof}

Consider the case $\sigma < \infty$. According to Lemma \ref{lem:finite}, $\mathcal G_2^{\gamma, p}$ is finite. Thus we can decompose $[a,b]$ into finite closed subintervals $\big([a_k,b_k]\big)_{k=1}^N$ of $[a,b]$ such that $\mathcal G_2^{\gamma, p}\cap [a_k,b_k]$ contains at most one element and  $|b_k-a_k|$ is sufficiently small. Using the partition of unity argument, we can reduce the matters to only considering $\gamma: [a,b]\to \R^2$ for which $\mathcal G_1^{\gamma, p} = [a,b]$ and $\mathcal G_2^{\gamma, p}$ is either empty or a singleton set. Using translation, we may assume that $0\in [a,b]$, and moreover, if $\mathcal G_2^{\gamma, p}\neq \emptyset$, then we can further assume that $\mathcal G_2^{\gamma, p} = \{0\}$.
Also, we may assume that the length of the interval $b-a$ is sufficiently small, in that, $|a|,|b|\le\eps$ for a sufficiently small $\eps>0$.

Consequently, we deduce that \eqref{e:main} would follow once we prove the following.

\begin{prop}\label{prop:main}
    Let $h_0>0$ be a sufficiently small number and $K = K_P\times K_F\subset \R^2\times \R^2$ be a compact set. Let $p = p(x,\xi)$ be an admissible real-valued symbol on $T^*\R^2$ and $\gamma:[a,b]\to K_P$ be a smooth curve in $\R^2$ such that $\mathcal G_1^{\gamma, p} = [a,b]$ and $|b-a|\le\eps$ for sufficiently small $\eps>0$. Suppose that there exists a unique smooth function $\xi:[a,b]\to K_F$ satisfying
    \begin{align*}
        p(\gamma(t),\xi(t)) = 0,\quad \dot\gamma(t) = \partial_\xi p(\gamma(t), \xi(t)),\quad t\in [a,b].
    \end{align*}
    Assume further that either $\sigma = \infty$, or $\sigma < \infty$ and $\mathcal G_2^{\gamma, p}$ contains at most one element.
    Then for every $O(h)$ quasimode $f$ of $P(h)$ satisfying the localization property \eqref{i:localized}, there exists a positive constant $C = C(\gamma,p)$ such that
    \begin{align}\label{e:l2u}
    \|f\|_{L^2(\gamma)}\le C h^{-\rho(2,\sigma)},\quad h\in (0,h_0].
    \end{align}
\end{prop}

\begin{rem}
    The extension of $\gamma$ to $\wt\gamma$ introduced at the end of Subsection \ref{ssec:basic} has only been employed to facilitate considering the order of contact $\sigma_{\gamma, p}(t)$ at the endpoints $t=a,b$ and does not play any role on the subsequent sections. We point out that the statement of Proposition \ref{prop:main} is irrelevant to the choice of the extension $\wt\gamma$.
\end{rem}

\section{Proof of Proposition \ref{prop:main}}\label{sec:redpropa}

In this section, we establish Proposition \ref{prop:main}. From now on, we occasionally write $P =P(h)$ for simplicity.

\subsection{Reduction to estimating an oscillatory integral} The first step toward proving Proposition \ref{prop:main} is to convert the bounds of quasimodes \eqref{e:l2u} into resolvent estimates. By the assumption that $f$ is a quasimode, there exists a $L^2$ normalized function $v$ such that $Pf = hv$. Since $p\in C_c^\infty(T^*\R^2)$ and real-valued, $P = p^w(x,hD)$ is a bounded self-adjoint operator. Hence the inverse $(P+ih)^{-1}$ exists, giving the expression $f = (h^{-1}P+i)^{-1}(if+v)$. Recalling the localization condition on $f$, we write
\[
f = \chi^w(x,hD)(h^{-1}P+i)^{-1}(if+v) + O_{\mathcal S}(h^\infty).
\]
Thus the proposition follows once the inequality
\begin{align}\label{e:l2v}
    \big\|\chi^w(x,hD)(h^{-1}P+i)^{-1}\big\|_{L^2\to L^2(\gamma)}\lesssim h^{-\rho(2,\sigma)}
\end{align}
is established.

To obtain \eqref{e:l2v}, we employ a standard approach, expressing the resolvent as an integral involving the propagator $e^{ih^{-1}Ps}$.
Applying the Fourier inversion formula, we can write
\begin{align*}
    (h^{-1}P+i)^{-1} &= \frac{h^{-1}P}{h^{-2}P^2+1} - \frac{i}{h^{-2}P^2+1} \\
    &=\int \mathcal F\Big(\frac{x}{x^2+1}\Big)(s)e^{ih^{-1}Ps} ds -i\int \mathcal F\Big(\frac{1}{x^2+1}\Big)(s)e^{ih^{-1}Ps} ds,
\end{align*}
where $\mathcal F$ denotes the Fourier transform. In the second line, we use the fact that $P$ is a compact operator, implying that there exists the spectral decomposition of $P$ (see \cite{RS80, Zwo12}). Since $\mathcal F((x^2+1)^{-1})(s)=\pi e^{-|s|}$ and $\mathcal F(x(x^2+1)^{-1})(s)=i\partial_s\mathcal F((x^2+1)^{-1})(s)$, the derivatives of these integrands are bounded by $C(1+|s|)^{-N}$ for $s\neq 0$, $N\in\N$.
Hence, it suffices to show that the estimate
\begin{align}\label{e:chiphi}
    \big\|\chi^w(x,hD)\widecheck\varphi(h^{-1}P)f\big\|_{L^2(\gamma)}\lesssim h^{-\rho(2,\sigma)}\|f\|_2
\end{align}
holds for every $\varphi\in C^\infty(\R\setminus\{0\})$ such that $|\varphi^{(m)}(s)|\lesssim (1+|s|)^{-N}$. Here, $\widecheck\varphi$ means the inverse Fourier transform of $\varphi$.

A simple decomposition argument further reduces the problem to showing \eqref{e:chiphi} only for $\varphi\in C^\infty(\R\setminus\{0\})$ supported in $(-\eps,\eps)$, where $\eps>0$ is a sufficiently small constant. To achieve this, let $\varsigma\in C_c^\infty(\R)$ such that $\varsigma\equiv 1$ on $[-1/4,1/4]$, $\supp(\varsigma)\subset (-1,1)$, and $\sum_{k\in\Z}\varsigma(s-k) = 1$ for $s\in\R$. Then we write
\begin{align}\nonumber
    \widecheck\varphi(h^{-1}P)f &= \sum_{k\in\Z}\int \varsigma\big(\eps^{-1}(s-k\eps)\big)\varphi(s) e^{\frac{i}{h}Ps}f\,ds \\\nonumber
    &= \sum_{k\in\Z} \int\varsigma(\eps^{-1}s)\varphi(s+k\eps) e^{\frac{i}{h} Ps}\big(e^{\frac{i}{h}k\eps P}f\big) ds.
\end{align}
Since $|\varphi^{(m)}(s+k\eps)|\lesssim (1+|k|\eps)^{-N}$ on $\supp(\varsigma(\eps^{-1}\cdot))$ and $\|e^{ih^{-1}k\eps P}\|_{2\to 2} = 1$, the estimate
\[
\bigg\|\int\varsigma(\eps^{-1}s)\varphi(s+k\eps) \chi^w(x,hD) e^{\frac{i}{h} Ps}\big(e^{\frac{i}{h}k\eps P}f\big) ds\bigg\|_{L^2(\gamma)}\lesssim (1+|k|\eps)^{-N} h^{-\rho(2,\sigma)}\|f\|_2
\]
follows once \eqref{e:chiphi} is established under the additional assumption $\supp(\varphi)\subset (-\eps,\eps)$.
Summing this over $k\in\Z$ verifies \eqref{e:chiphi} for general $\varphi\in C^\infty(\R\setminus\{0\})$ satisfying $|\varphi^{(m)}(s)|\lesssim (1+|s|)^{-N}$.

Next, we make use of a $TT^*$-argument. As a result, proving \eqref{e:chiphi} is reduced to showing that there exists a constant $C=C(\varphi, p, \gamma)>0$ such that
\begin{align}\label{e:chi2phi}
    \big\|\chi^w(x,hD)\widecheck\varphi(h^{-1}P)\chi^w(x,hD)\big\|_{L^2(\gamma)\to L^2(\gamma)}\le C h^{-2\rho(2,\sigma)}
\end{align}
holds for $\varphi\in C^\infty(\R\setminus\{0\})$ with $\supp(\varphi)\subset (-\eps,\eps)$. To verify this, we make use of a classical kernel representation formula for the propagator $e^{ih^{-1}Ps}$. Note that
\begin{align}\label{i:2chivphi}
    \chi^w(x,hD)\widecheck\varphi(h^{-1}P)\chi^w(x,hD) = \int \varphi(s) \chi^w(x,hD)e^{ih^{-1}Ps}\chi^w(x,hD) ds.
\end{align}

\begin{prop}\label{prop:asym1}
    There exists $s_0>0$ independent of $h$, such that for $0\le s\le s_0$,
    \begin{align*}
        \chi^w(x,hD)e^{ih^{-1}Ps}&\chi^w(x,hD) f(x)  \\
        &= \frac{1}{h^2}\int e^{\frac ih(\phi(s,x,\eta)-\inp{y}{\eta})}b(s,x,\eta;h)f(y)dyd\eta + E(s)f(x),
    \end{align*}
    where $\phi$ is a solution to
    \begin{align}\label{i:phisol}
        \partial_s\phi(s,x,\eta) + p(x,\partial_x\phi(s,x,\eta)) = 0,\quad \phi(0,x,\eta) = \inp{x}{\eta},
    \end{align}
    $b\in C_c^\infty(\R\times T^*\R^2)$ such that $\supp(b(s,\cdot,\cdot))\subset K$, $|b(s,\gamma(0),\xi(0))|>0$ for every $0\le s\le s_0$. Also, $E(s) = O(h^N) : \mathcal S'\to \mathcal S$ with a sufficiently large $N$. Furthermore, $\phi$ satisfies the relation
    \begin{align}\label{i:relphi}
        (x,\partial_x\phi(s,x,\eta)) = \kappa_s(\partial_\eta\phi(s,x,\eta), \eta).
    \end{align}
\end{prop}

\begin{proof}
    By a standard argument solving the eikonal equation (see \cite[Proposition 4.2]{KTZ07}), it follows that
    \begin{align*}
        e^{ih^{-1}Ps}&\chi^w(x,hD) f(x) = h^{-2}\int e^{\frac ih(\phi(s,x,\eta)-\inp{y}{\eta})}c(s,x,\eta;h)f(y)dyd\eta + \wt E(s)f(x),
    \end{align*}
    where $\phi$ is a solution of \eqref{i:phisol}, $c\in C_c^\infty(\R\times T^*\R^2)$ such that $\supp(c(s,\cdot,\cdot))\subset K$ and $|c(s,\gamma(0),\xi(0))|>0$ for $0\le s\le s_0$ with some $s_0>0$, and $\wt E(s)f$ is the error term satisfying $\|\wt E(s)f\|_2 = O(h^\infty)$.
    Using this, we express 
    \begin{align}\nonumber
        \chi^w(x,hD)&e^{ih^{-1}Ps}\chi^w(x,hD) f(x) \\\label{i:2chiprop}
        &= \frac{1}{h^4}\int e^{\frac{i}{h}(\inp{x-y}{\zeta}+\phi(s,y,\eta)-\inp{w}{\eta})} c(s,y,\eta;h) \chi(\frac{x+y}{2},\zeta)f(w) dyd\zeta d\eta dw 
    \end{align}
    modulo an error in $O_{L^2}(h^\infty)$.
    By a calculation, the first and second-order derivatives of the phase function $\inp{x-y}{\zeta}+\phi(s,y,\eta)-\inp{w}{\eta}$ in the variable $(y,\zeta)$ are expressed as follows.
    \begin{align*}
    \partial_{(y,\zeta)}\big(\inp{x-y}{\zeta}+\phi(s,y,\eta)-\inp{w}{\eta}\big) &= \big(\partial_y\phi(s,y,\eta) - \zeta, x-y\big), \\
        \partial_{(y,\zeta)}^2\big(\inp{x-y}{\zeta}+\phi(s,y,\eta)-\inp{w}{\eta}\big) &= \begin{pmatrix}
            \partial_y^2\phi(s,y,\eta) & -I \\
            -I & 0
        \end{pmatrix}.
    \end{align*}
    Note that the matrix $\partial_{(y,\zeta)}^2\big(\inp{x-y}{\zeta}+\phi(s,y,\eta)-\inp{w}{\eta}\big)$ is invertible, thus the phase has a nondegenerate stationary point at $(y_0, \zeta_0) = (x, \partial_x\phi(s,x,\eta))$. Applying the stationary phase method (see \cite [p.344]{Ste93}) yields
    \begin{align*}
        \frac{1}{h^2} \int &e^{\frac{i}{h}\big(\inp{x-y}{\zeta}+\phi(s,y,\eta)-\inp{w}{\eta}\big)} c(s,y,\eta;h) \chi(\frac{x+y}{2},\zeta) dyd\zeta \\
        & = \sum_{n=0}^N h^n e^{\frac{i}{h}(\phi(s,x,\eta)-\inp{w}{\eta})}c_n(s,x,\eta;h) + E(s,x,\eta;h)
    \end{align*}
    for $N\in\N$, where $E = O(h^{N+1})$ is the error term, $c_n\in C_c^\infty(\R\times T^*\R^2)$ such that the support of $(x,\eta)\mapsto c_n(s,x,\eta;h)$ is contained in $K$ for $s\in\R$, and $|c_0(s,\gamma(0),\xi(0))|>0$. Choosing a sufficiently large $N$, we set
    \[
    b(s,x,\eta;h) := \sum_{n=0}^N h^n c_n(s,x,\eta;h).
    \]
    By \eqref{i:2chiprop}, we have
    \begin{align*}
        \chi^w(x,hD)&e^{ih^{-1}Ps}\chi^w(x,hD) f(x)\\
        &= \frac{1}{h^2}\int e^{\frac{i}{h}(\phi(s,x,\eta)-\inp{w}{\eta})} b(s,x,\eta;h) f(w) dwd\eta + E(s)f(x)
    \end{align*}
    with $E(s) : L^2\to L^2$ such that $E(s)f \in O_{L^2}(h^{N+1})$. This completes the proof of Proposition \ref{prop:asym1}. 
\end{proof}

Recall that $\gamma$ is parametrized by $\gamma: [a,b]=:I\to \R^2$. Let
\[
T f(u) := \frac{1}{h^2} \int \varphi(s) e^{\frac{i}{h}(\phi(s,\gamma(u),\eta)-\inp{\gamma(v)}{\eta})} b(s,\gamma(u),\eta;h)f(v) d\eta ds dv,\quad u,v\in I.
\]
By Proposition \ref{prop:asym1}, \eqref{e:chi2phi} is a direct consequence of the estimate 
\begin{align}\label{est:TI}
    \|T\|_{L^2(I)\to L^2(I)}\lesssim h^{-2\rho(2,\sigma)}.
\end{align}

\subsection{Dyadic decomposition in time}
To estimate the $L^2(I)$ operator norm of $T$, we consider a dyadic decomposition of the kernel in $s$ away from zero. Let $\psi$ be a cutoff function on $\R$ such that $\psi\equiv 1$ on $[-5/4,-3/4]\cup[3/4,5/4]$, $\supp(\psi)\subset (-2,-1/2)\cup(1/2,2)$, and $\sum_{k\in\Z}\psi(2^k s) = 1$ for $s\in\R\setminus\{0\}$.
For $k\ge 0$, we define the operator $T_k$ by
\[
T_k f(u) := \frac{1}{h^2} \int \varphi(s)\psi(2^k s) e^{\frac{i}{h}(\phi(s,\gamma(u),\eta)-\inp{\gamma(v)}{\eta})} b(s,\gamma(u),\eta;h) f(v) d\eta ds dv.
\]
Clearly, $Tf$ is decomposed as $Tf = \sum_{k\ge 0}T_k f$. We proceed to establish the following estimates for the dyadic pieces. 

\begin{prop}\label{prop:esttk}
    Let $\sigma\ge 1$ and $k\ge 0$. Suppose that $p$ is an admissible real-valued symbol. Then we have the following.
    \begin{itemize}[topsep = -3pt]
        \item [i)] Assume that $\sigma < \infty$. Then we have
        \begin{itemize}[topsep=5pt]
            \item [a)] If $2^{-k} \lesssim h^{\frac{1}{2\sigma+1}}$, then
            \begin{align*}
        \big\|T_k f\big\|_{L^2(I)} \lesssim \begin{cases}
            2^{-k}h^{-1}\|f\|_2, & \text{ if}\ \  2^{-k}\lesssim h, \\
            2^{-\frac k2} h^{-\frac12}\|f\|_2, & \text{ if}\ \  h\lesssim 2^{-k}\lesssim h^{\frac{1}{2\sigma+1}}.
        \end{cases}
    \end{align*}
            \item [b)] If $2^{-k}\gtrsim h^{\frac{1}{2\sigma+1}}$, then
            \begin{align*}
        \big\|T_k f\big\|_{L^2(I)} \lesssim \begin{cases}
            2^{\sigma k}\|f\|_2, & \text{ if $\sigma$ is odd}, \\
            h^{-\frac14} 2^{\frac{2\sigma-1}{4}k}\|f\|_2, & \text{ if $\sigma$ is even}.
        \end{cases}
    \end{align*}
        \end{itemize}
    \item [ii)] Assume that $\sigma = \infty$. Then we have
    \begin{align*}
        \big\|T_k f\big\|_{L^2(I)} \lesssim \begin{cases}
            2^{-k}h^{-1}\|f\|_2, & \text{ if}\ \  2^{-k}\lesssim h, \\
            2^{-\frac k2} h^{-\frac12}\|f\|_2, & \text{ if}\ \  h\lesssim 2^{-k}\lesssim 1.
        \end{cases}
    \end{align*}
    \end{itemize}
\end{prop}

Via the triangle inequality, it can be easily checked that Proposition \ref{prop:esttk} implies the required estimate \eqref{est:TI}.

As to be seen later, the kernel $T_k(u,v)$ rapidly decays when the distance $|u-v|$ is not comparable to $2^{-k}$. This means that the contributions of $T_k f$ are significant only when $|u-v|\sim 2^{-k}$. It suggests making a decomposition of the kernel with respect to $|u-v|$ to estimate $T_k f$ efficiently.
For $l\ge 0$, let $T_{k,l}$ be an operator from $L^2(I)$ to $L^2(I)$ whose kernel is given by
\[
T_{k,l}(u,v) = T_k(u,v) \psi(2^l|u-v|).
\]
Then it is clear that
\[
T_k f(u) = \sum_{l\ge 0} T_{k,l} f(u).
\]
To estimate each component, we further decompose the $u,v$-spaces into $O(2^l)$ intervals of length $\sim 2^{-l}$. For $l\ge 0$, $n\in\Z$, we define
\begin{align*}
    \varsigma_{l,n}(u) := \varsigma(2^{l+6}u - n),\quad \varsigma'_{l,n}(v) := \varsigma(2^{l-4}v - 2^{-10}n) - \varsigma(2^{l+2}v - 2^{-4}n).
\end{align*}
It is easy to see that the relations
\begin{align*}
    &\sum_{-2^l\le n\le 2^l} \varsigma_{l,n}(u) \equiv 1\ \text{on $I$}, \\
    &\dist(\supp(\varsigma_{l,n}), \supp(\varsigma'_{l,n})) \sim 2^{-l}, \\
    &\varsigma'_{l,n}(v) = 1\ \text{if } u\in \supp(\varsigma_{l,n}),\ \psi(2^l|u-v|)\neq 0
\end{align*}
are satisfied. This implies
\[
\sum_{-2^l\le n\le 2^l} \varsigma_{l,n}(u)\varsigma'_{l,n}(v) = 1
\]
on the support of $(u,v)\mapsto \psi(2^l|u-v|)$, allowing us to decompose
\begin{align*}
    T_{k,l}f(u) = \sum_{-2^l\le n\le 2^l} T_{k,l}^n f(u) := \sum_{-2^l\le n\le 2^l} \int T_{k,l}(u,v)\varsigma_{l,n}(u)\varsigma'_{l,n}(v) f(v) dv.
\end{align*}

To sum the estimates for $T_{k,l}^n f$ in $n$, we will make use of the elementary result below.

\begin{lem}\label{lem:sumn}
    Suppose that for $-2^l\le n\le 2^l$, there exists a constant $ C(h, \sigma,k,l,n)>0$ such that the estimate 
    \begin{align}\label{e:tkln}
        \|T_{k,l}^n f\|_{L^2(I)} \le C(h,\sigma,k,l,n) \|f\|_{L^2(I)}
    \end{align}
    holds. Then we have
    \begin{align*}
        \Big\|\sum_{-2^l\le n\le 2^l} T_{k,l}^n f\,\Big\|_{L^2(I)} \lesssim \max_{n : -2^l\le n\le 2^l} C(h, \sigma,k,l,n)\,\|f\|_{L^2(I)}.
    \end{align*}
\end{lem}
\begin{proof}
    We note that $\supp(\varsigma_{l,n})\cap \supp(\varsigma_{l,n'}) = \emptyset$ unless $|n-n'|\le 4$. This implies
    \begin{align*}
        \Big\|\sum_{-2^l\le n\le 2^l} T_{k,l}^n f\,\Big\|_{L^2(I)}^2 \lesssim \sum_{-2^l\le n\le 2^l} \big\|T_{k,l}^n f\big\|_{L^2(I)}^2.
    \end{align*}
    By \eqref{e:tkln} and the fact that the sets $\supp(\varsigma'_{l,n})$ has finite overlap, the sum on the right-hand side is bounded by a constant times
    \begin{align*}
        \Big(\max_{n : -2^l\le n\le 2^l} C(h,\sigma,k,l,n)\Big)^2&\,\sum_{ -2^l\le n\le 2^l}\|\varsigma'_{l,n}f\|_{L^2(I)}^2 \\ &\lesssim \Big(\max_{n : -2^l\le n\le 2^l} C(h,\sigma,k,l,n)\Big)^2 \|f\|_{L^2(I)}^2,
    \end{align*}
    which immediately gives the desired result.
\end{proof}

\subsection{Properties of the phase $\phi$}
In this subsection, we collect some properties of $\phi$ which will frequently be used later.

\begin{prop}
    Let $\phi$ be as in Proposition \ref{prop:asym1}. Then the following hold.
    \begin{align}\nonumber
    &\partial_s\partial_\eta \phi(s,x,\eta) = - \partial_x^\intercal\partial_\eta \phi(s,x,\eta) \partial_\xi p(x, \partial_x\phi(s,x,\eta)), \\[2pt]
    \label{i:phids}
    &\partial_x^\intercal\partial_\eta \phi(0,x,\eta) = I,\quad \partial_\eta^2\phi(0,x, \eta) = 0, \\[2pt]
    \label{i:1dphix}
    &\partial_x\phi(s,x,\eta)=\eta+O(s),\\[2pt]
    \label{i:1dphiy}
    &\partial_\eta\phi(s,x,\eta)  = x-s\,\partial_\xi p(x,\eta) + O(s^2).
    \end{align}
\end{prop}

\begin{proof}
    The first identity is a direct consequence of differentiating the first identity in \eqref{i:phisol} with respect to $\eta$. The second one \eqref{i:phids} follows from the second identity in \eqref{i:phisol}. 
    The third one is deduced from Taylor's theorem and the fact that $\partial_x\phi(0,x,\eta)=\eta$.
    On the other hand, using \eqref{i:phisol} we have $\partial_s\partial_\eta\phi(0,x,\eta)=-\partial_\xi p(x,\eta)$, $\partial_\eta \phi(0,x,\eta)=x$.
    From this and Taylor's theorem, the last formula \eqref{i:1dphiy} follows.
\end{proof}

\subsection{Estimates for nonstationary part}
The $L^2$ bounds for $T^n_{k,l} f$ are obtained by dividing into two cases in $k$. First, we consider the easier case, $2^{-k}\lesssim h$.
In this case, we can obtain favorable bounds for the pieces $T_{k,l}^n f$ using integration by parts.

\begin{lem}\label{lem:2jsmall}
    Let $k,l\ge 0$, $N\in\N$. Suppose that $2^{-k}\lesssim h$. Then there exists a constant $C = C(p,\gamma)>0$ such that
    \begin{align*}
            \|T_{k,l}^n f\|_{L^2(I)}\lesssim \begin{cases}
                h^{-2}2^{-k-l}\|f\|_2, & \text{if } 2^{-l}\le Ch, \\
                h^{-2}2^{-k-l}(h2^l)^N\|f\|_2, & \text{if } 2^{-l}\ge Ch
            \end{cases}
        \end{align*}
        holds for $-2^l\le n\le 2^l$.
\end{lem}

\begin{proof}
    First, we assume $2^{-l}\lesssim h$. We note that the kernel of $T_{k,l}^n$ has the trivial bound
    \begin{align*}
        \big|T_{k,l}^n(u,v)\big|&\lesssim h^{-2}\int \Big|\varphi(s)\psi(2^k s) b(s,\gamma(u),\eta)\varsigma_{l,n}(u)\varsigma'_{l,n}(v)\Big|dsd\eta \\
        &\lesssim h^{-2}2^{-k}\varsigma_{l,n}(u)\varsigma'_{l,n}(v),
    \end{align*}
    which, combined with H\"older's inequality, gives
    \[
    \|T_{k,l}^n f\|_{L^2(I)}\lesssim h^{-2}2^{-k} 2^{-l}\|f\|_2.
    \]
    Now suppose that $2^{-l}\gg h$. Making the change of variable $s\to 2^{-k}s$, we write
    \begin{align*}
        T_{k,l}^n(u,v) = \frac{2^{-k}}{h^2}\int \varphi(2^{-k}s) \psi(s) &e^{\frac{i}{h}(\phi(2^{-k} s,\gamma(u),\eta) - \inp{\gamma(v)}{\eta})}\\
        &\qquad\qquad \times b(2^{-k}s,\gamma(u),\eta)\varsigma_{l,n}(u)\varsigma'_{l,n}(v) dsd\eta.
    \end{align*}
    Using \eqref{i:1dphiy}, we calculate
    \begin{align}\nonumber
        \partial_\eta\big(\phi(2^{-k}s,x,\eta) - \inp{y}{\eta}\big) & = \partial_\eta\phi(2^{-k}s,x,\eta) - y \\\label{i:1dphiy2}
        & = x-y-2^{-k}s\,\partial_\xi p(x,\eta) + O(2^{-2k}).
    \end{align}
    On the other hand, using \eqref{i:phisol} and \eqref{i:1dphix} gives
    \begin{align}\label{i:1dphis}
        \partial_s\big(\phi(2^{-k}s,x,\eta) - \inp{y}{\eta}\big) = -2^{-k}p(x,\eta) + O(2^{-2k}),\quad x,y\in \mathbb R^2.
    \end{align}
    By the condition (A1),
    \begin{align}\label{e:lbppxi}
        |p(x,\eta)| + |\partial_\xi p(x,\eta)|\ge C,\quad (x,\eta)\in K.
    \end{align}
    The same condition also guarantees that $|\dot\gamma(u)| = |\partial_\xi p(\gamma(u),\xi(u))|\ge C$.
    Consequently, substituting $x = \gamma(u)$, $y = \gamma(v)$ into \eqref{i:1dphiy2}, \eqref{i:1dphis}, we obtain
\begin{align}\label{b:lbphi}
    \big|\partial_{(s,\eta)}\big(\phi(2^{-k} s,\gamma(u),\eta) - \inp{\gamma(v)}{\eta}\big)\big| \gtrsim \begin{cases}
        2^{-k}, & \text{ if } 2^{-l}\ll 2^{-k}, \\
        2^{-l}, & \text{ if } 2^{-l}\gg 2^{-k}.
    \end{cases}
\end{align}
    Also, it is immediate that $|\partial_{(s,\eta)}^\beta (\phi(2^{-k}s,x,\eta) - \inp{y}{\eta})|\le C_ \beta$ for $\beta\in \N_0^3$. Thus, repeated integration by parts yields
    \begin{align*}
        |T_{k,l}^n(u,v)|\lesssim h^{-2} 2^{-k} (h2^l)^N \varsigma_{l,n}(u)\varsigma'_{l,n}(v)
    \end{align*}
    provided $2^{-l}\gg 2^{-k}$. Using H\"older's inequality yields the second estimate of the lemma.
\end{proof}

Combining Lemma \ref{lem:sumn} and \ref{lem:2jsmall}, we can obtain the first estimate of Proposition \ref{prop:esttk}. By Lemma \ref{lem:sumn}, we have
\begin{align*}
    \|T_{k,l}f\|_{L^2(I)}\lesssim 
                h^{-2}2^{-k-l}\min(1,(h2^l)^N)\|f\|_2.
\end{align*}
Summing the estimate over $l$, the required estimate follows.

Now, we look for the case $2^{-k}\gtrsim h$.
To address the case, we divide into two subcases: \textit{i) $2^{-l}\sim 2^{-k}$, ii) $2^{-l}\ll 2^{-k}$ or $2^{-l}\gg 2^{-k}$.} As mentioned earlier, the main contribution arises from the former case. We first consider the second case, which can be dealt with by using the integration by parts.

Since the higher order derivatives of the phase function in $\eta$ are uniformly bounded, from integration by parts and \eqref{b:lbphi}, we have
\begin{align*}
    |T_{k,l}^n(u,v)| \lesssim
        h^{-2}2^{-k}h^N \min(2^{kN},2^{lN})\,\varsigma_{l,n}(u)\varsigma'_{l,n}(v)
\end{align*}
provided that $2^{-l}\ll 2^{-k}$ or $2^{-l}\gg 2^{-k}$. Using H\"older's inequality, we get
\begin{align*}
            \|T_{k,l}^n f\|_{L^2(I)}\lesssim \begin{cases}
                h^{-2}2^{-k-l}(h 2^k)^N, & \text{if } 2^{-l}\ll 2^{-k}, \\
                h^{-2}2^{-k-l}(h2^l)^N, & \text{if } 2^{-l}\gg 2^{-k}.
            \end{cases}
        \end{align*}
Applying Lemma \ref{lem:sumn} shows that $T_{k,l} f$ satisfies the same estimate as above. Summing this over $l$, we obtain
\begin{align}
\begin{aligned}\label{e:easy}
    \bigg\|\sum_{l : 2^{-l}\ll 2^{-k} \text{ or } 2^{-l}\gg 2^{-k}} T_{k,l} f\bigg\|_{L^2(I)} &\lesssim h^{-2} 2^{-2k} (h2^k)^N \\
    &\lesssim 1
\end{aligned}
\end{align}
for $N\in\N$ because $2^k h\lesssim 1$. Note that the resulting bound is less than $2^{-k/2}h^{-1/2}$ and $2^{k\sigma}$, which are the bounds we need to obtain to prove Proposition \ref{prop:esttk}.

\subsection{Asymptotic expansion of kernel}
From now on, we assume that $2^{-k}\gtrsim h$, $2^{-l}\sim 2^{-k}$. An additional difficulty arises in handling this case compared to the previous ones because the derivatives of $\phi(s,\gamma(u),\eta) - \inp{\gamma(v)}{\eta}$ in $s,\eta$ may vanish at some $(s,\gamma(u),\eta)\in \supp(b)$, unlike the previous cases. Recall that there exists a smooth curve $\xi : I\to \R^2$ such that
\begin{align}\label{r:gamxiu}
    p(\gamma(u), \xi(u)) = 0, \quad \dot\gamma(u) = \partial_\xi p(\gamma(u), \xi(u))
\end{align}
for every $u\in I$. Also, let 
\[
\mathcal A_{l,n} := \supp(\varsigma_{l,n})\times \supp(\varsigma'_{l,n}).
\]

\begin{lem}\label{lem:seta}
    Let $(u,v)\in \mathcal A_{l,n}$. Then there exists a unique $(s(u,v), \eta(u,v))\in \R\times \R^2$ such that
    \begin{align}\label{i:seta0}
        &\partial_{(s,\eta)}\big(\phi(s,\gamma(u),\eta) - \inp{\gamma(v)}{\eta}\big)\big|_{(s,\eta) = (s(u,v), \eta(u,v))} = 0, \\\label{r:dseta}
        & s(u,v)\in B(u-v, C2^{-2k}),\quad \eta(u,v)\in B(\xi(u), C2^{-k})
    \end{align}
    with a constant $C>0$. Furthermore, the map $(u,v)\mapsto (s(u,v), \eta(u,v))$ is smooth on $\mathcal A_{l,n}$ and satisfies the bounds
    \begin{align}\label{b:dseta}
        \big|\partial_{(u,v)}^\beta s(u,v)\big|\le C_\beta,\quad \big|\partial_{(u,v)}^\beta \eta(u,v)\big|\le  2^kC_\beta
    \end{align}
    for every $\beta\in \N^2$ with a constant $C_\beta>0$.
\end{lem}

\begin{proof}
    Using \eqref{i:phisol}, we calculate
    \[
    \partial_{(s,\eta)}\big(\phi(s,\gamma(u),\eta) - \inp{\gamma(v)}{\eta}\big) = \big(-p(\gamma(u),\partial_x\phi(s,\gamma(u),\eta)), \partial_\eta \phi(s,\gamma(u),\eta)- \gamma(v)\big).
    \]
    Now we let
    \begin{align*}
        \Psi(u,v,s',\eta) := \big(-p(\gamma(u),\partial_x\phi(2^{-k}s',\gamma(u),\eta)),\,& 2^{k}(\partial_\eta \phi(2^{-k}s',\gamma(u),\eta)- \gamma(v))\big).
    \end{align*}
    Using Taylor's theorem together with \eqref{i:1dphix}, \eqref{i:1dphiy} gives
    \begin{align}
    \label{hessianofphase}
        \partial_{(s',\eta)}\Psi(u,v,s',\eta) &=\begin{pmatrix}
            0 & \mathrm v^\intercal \\
            \mathrm v & \mathbf N
        \end{pmatrix}
        + O(2^{-k}) \\
        \nonumber
        &=:\mathbf M+O(2^{-k})
    \end{align}
    where $\mathrm v=-\partial_\xi p(\gamma(u),\eta)$ and $\mathbf N=s'\partial_\xi^2 p(\gamma(u),\eta)$. From \eqref{e:lbppxi} and \eqref{r:gamxiu}, it follows that $|\mathrm v|\gtrsim 1$ if $\eta$ and $\xi(u)$ are sufficiently close. By the condition (A2), we have
    \begin{equation}
    \label{ineq:vMv}
        \inp{\mathrm v_\perp} {\mathbf N\mathrm v_\perp} \gtrsim 1
    \end{equation}
where $\mathrm v_\perp$ denotes the unit vector in $\R^2$ such that $\langle \mathrm v,\mathrm v_\perp\rangle = 0$. We note that
\begin{align}\label{i:Mprod}
\mathbf M \cdot
    \begin{pmatrix}
        1 & 0 & 0\\
        0 & \mathrm v &\mathrm v_\perp 
    \end{pmatrix}
    =
    \begin{pmatrix}
        0 & |\mathrm v|^2 & 0 \\
        \mathrm v & \mathbf N\mathrm v &\mathbf N \mathrm v_\perp
    \end{pmatrix},
\end{align}
Combining with \eqref{ineq:vMv} yields $|\det \mathbf M|=|\det     \begin{pmatrix}
        0 & |\mathrm v|^2 & 0 \\
        \mathrm v & \mathbf N\mathrm v &\mathbf N \mathrm v_\perp
    \end{pmatrix}|\gtrsim 1$.
Thus the matrix $\mathbf M+O(2^{-k})$ is invertible and has a determinant whose absolute value is comparable to $1$ provided that $|\eta-\xi(u)|$ is sufficiently small. From \eqref{i:Mprod}, it also follows that $\|\mathbf M^{-1}\|\lesssim 1$. Moreover, a simple calculation gives
\begin{align}\label{e:Psihd}
    \big|\partial_{(u,s',\eta)}^\alpha \partial_v^\beta \Psi(u,v,s',\eta)\big|\le \begin{cases}
        C_\alpha, & \beta = 0, \\
        2^k C_{\alpha, \beta}, & \beta\neq 0,
    \end{cases} 
\end{align}
for $(u,v,\eta)\in \mathcal A_{l,n}\times \R^2$, $|s'|\lesssim 1$, and $\alpha\in \N_0^4$, $\beta\in \N_0$ such that $|\alpha|\ge 1$.

On the other hand, using Taylor's theorem and $\dot\gamma(u)=\partial_\xi p(\gamma(u),\xi(u))$, we write
    \begin{align*}
        \gamma(u) - \gamma(v) = (u-v)\partial_\xi p(\gamma(u), \xi(u)) + O(|u-v|^2),
    \end{align*}
    which implies that, if $s'=2^k(u-v)$, $\eta=\xi(u)$,
    \[
    2^k\big(\partial_\eta \phi(2^{-k}s',\gamma(u),\eta)- \gamma(v)\big) = O(2^{-k}).
    \]
    Furthermore, using Taylor's theorem, \eqref{i:1dphix}, and the identity $p(\gamma(u),\xi(u)) = 0$, we have
    \[
    p(\gamma(u), \partial_x\phi(2^{-k}s', \gamma(u), \eta)) = O(2^{-k}),\quad \text{for $(s', \eta) = (2^k(u-v), \xi(u))$}.
    \]
    As a result, $\Psi(u,v,2^k(u-v), \xi(u)) = O(2^{-k})$. Recall that the factor $2^{-k}$ is assumed to be sufficiently small. From the previous discussion and the implicit function theorem, it follows that there exists a unique smooth function $(s'(u,v),\eta(u,v))$ defined on $\mathcal A_{l,n}$ such that
    \begin{gather*}
    (s'(u,v),\eta(u,v))\in B(2^k(u-v), C2^{-k})\times B(\xi(u), C2^{-k}), \\
    \Psi(u,v,s',\eta)\big|_{s'=s'(u,v), \eta = \eta(u,v)} = 0,
    \end{gather*}
    and for every $\beta\in \N_0^2$ satisfying $|\beta|\ge 1$,
    \begin{align}\label{b:bds'ze}
        \big|\partial_{(u,v)}^\beta s'(u,v)\big|,\ \big|\partial_{(u,v)}^\beta \eta(u,v)\big|\le 2^kC_\beta.
    \end{align}
    Now let 
    \[
    s(u,v):= 2^{-k}s'(u,v).
    \]
    It is clear that the functions $s(u,v), \eta(u,v)$ satisfy all the assertions of the proposition. This completes the proof.
\end{proof}

Let $\tilde s(s) := 2^{-k}s + s(u,v)$, $\tilde\eta(\eta):= \eta+ \eta(u,v)$. We define
\[
\wt\phi(u,v,s,\eta) := \phi(\tilde s(s),\gamma(u),\tilde\eta(\eta)) - \inp{\gamma(v)}{\tilde\eta(\eta)}.
\]
We note that $\partial_{(s,\eta)}\wt\phi(u,v,0,0) = 0$. It can be shown that $(u,v,0,0)$ is a non-degenerate stationary point of $\wt\phi$. Indeed,
\begin{align}\label{i:wtphipsi}
    \partial_{(s,\eta)}\wt\phi(u,v,s,\eta) = 2^{-k}\Psi(u,v,s+s'(u,v), \tilde\eta(\eta)),
\end{align}
where $\Psi$, $s'(u,v)$ are as in the proof of Lemma \ref{lem:seta}. Hence \eqref{hessianofphase} and the fact $|\det \mathbf M|\gtrsim 1$ give
\[
|\det\big(2^k\partial_{(s,\eta)}^2\wt\phi(u,v,0,0)\big)| \sim 1.
\]
Also, by \eqref{e:Psihd} and \eqref{i:wtphipsi}, the derivatives of $ \wt\phi(u,v,s,\eta)$ in $s,\eta$ satisfy
\[
2^k|\partial_{(s,\eta)}^\beta\wt\phi(u,v,s,\eta)|\lesssim 1, \quad \beta\in \N_0^3,\ |\beta|\ge 3
\]
for $(u,v,\eta)\in \mathcal A_{l,n}\times \R^2$, $|s|\lesssim 1$.
Therefore, making use of \cite[Theorem 7.7.5]{Hor83}, we get the expression
\begin{align}\label{i:asymker}
    T_{k,l}^n(u,v) = h^{-\frac12}2^{\frac k2} B_{k,l}^n(u,v)e^{\frac{i}{h}\wt\phi(u,v,0,0)} + E^n_{k,l}(u,v),
\end{align}
where
\[
B_{k,l}^n(u,v) := \varsigma_{l,n}(u)\varsigma'_{l,n}(v)\psi(2^l|u-v|) b(s(u,v), \gamma(u), \eta(u,v))
\]
and $E^n_{k,l}$ is a smooth function supported in $\mathcal A_{l,n}$ and satisfying $\|E^n_{k,l}\|_{L^\infty(\R^2)}\lesssim h^{\frac12}2^{\frac32 k}$.
By \eqref{b:dseta},
\begin{align}\label{b:dbs}
    |\partial_{(u,v)}^\beta B_{k,l}^n(u,v)|\lesssim 2^{k|\beta|},\quad \beta\in \N_0^2.
\end{align}
Let $\Phi(u,v):= \wt\phi(u,v,0,0)$. Then we note
\begin{align}\label{i:defPhi}
    \Phi(u,v) = \phi(s(u,v),\gamma(u),\eta(u,v))-\inp{\gamma(v)}{\eta(u,v)}.
\end{align}

\noindent\textit{When $h\lesssim 2^{-k}\lesssim h^{\frac{1}{2\sigma+1}}$} \text{($h\lesssim 2^{-k}\lesssim 1$ if $\sigma = \infty$)}.
From \eqref{i:asymker}, it follows that $|T_{k,l}^n(u,v)|\lesssim h^{-\frac12}2^{\frac k2}\varsigma_{l,n}(u)\varsigma'_{l,n}(v)$. By the H\"older inequality,
\begin{align*}
    \|T_{k,l}^n f\|_{L^2(I)}\lesssim h^{-\frac12}2^{-\frac k2}\|f\|_2.
\end{align*}
Combining this with Lemma \ref{lem:sumn}, we get $\|T_{k,l} f\|_{L^2(I)}\le C h^{-\frac12}2^{-\frac k2}\|f\|_2$. Summing over $l : 2^{-l}\sim 2^{-k}$ yields
\begin{align*}
    \bigg\|\sum_{l : 2^{-l}\sim 2^{-k}} T_{k,l} f\bigg\|_{L^2(I)}\lesssim h^{-\frac12}2^{-\frac k2}\|f\|_2.
\end{align*}
The required estimate in Proposition \ref{prop:esttk} now follows from the above estimate and \eqref{e:easy}. We also remark that this verifies Proposition \ref{prop:esttk} for the case $\sigma = \infty$.

From now on, we assume that $2^{-k}\gtrsim h^{\frac{1}{2\sigma+1}}$ and $\sigma<\infty$. We first note that the error term $E^n_{k,l}$ is easily dealt with. Indeed, recall that $\|E^n_{k,l}\|_{L^\infty(\R^2)}\lesssim h^{\frac12}2^{\frac32 k}$. Since $E^n_{k,l}$ is supported in $\mathcal A_{l,n}$, using H\"older's inequality, we deduce that the $L^2(I)\to L^2(I)$ norm of the operator
\[
f\mapsto \int E^n_{k,l}(u,v) f(v) dv
\]
is bounded by $C h^{1/2}2^{k/2}$ for $C>0$.
Notice that
\[
h^{\frac12}2^{\frac k2} \lesssim \min(2^{\sigma k}, h^{-\frac14} 2^{\frac{2\sigma-1}{4}k}),
\]
which means the bound for $E^n_{k,l}$ is acceptable. Applying Lemma \ref{lem:sumn} to sum over $n$ and then summing over $l: l\sim k$, one can verify that $\sum_{n,l}E^n_{k,l}$ satisfies the desired estimate in Proposition \ref{prop:esttk}.

Now we define
\begin{align}\nonumber
    \mathcal T^n_{k,l} f(u) := \int B_{k,l}^n(u,v) e^{\frac ih \Phi(u,v)} f(v) dv.
\end{align}
By the above discussion, to prove Proposition \ref{prop:esttk}, it suffices to show the following.
\begin{prop}\label{prop:mainred}
    Assume that $\sigma<\infty$ and $2^{-l}\sim 2^{-k}\gtrsim h^{\frac 1{2\sigma+1}}$. Then
\begin{align}\label{e:ctnkl}
    \|\mathcal T^n_{k,l} f\|_{L^2} \lesssim \begin{cases}
        h^{\frac12}2^{\frac{2\sigma-1}{2}k}, & \text{ if $\sigma$ is odd,} \\
        h^{\frac14}2^{\frac{2\sigma-3}{4}k}, & \text{ if $\sigma$ is even.}
    \end{cases}
\end{align}
\end{prop}

\subsection{Handling the core part}
To verify Proposition \ref{prop:mainred}, we follow a standard $L^2$ argument. We write
\begin{align*}
    \|\mathcal T^n_{k,l} f\|_{L^2}^2 = \iint\bigg(\int e^{\frac ih(\Phi(u,w) - \Phi(u,v))} B_{k,l}^n(u,v) B_{k,l}^n(u,w) du\bigg) f(v)f(w) dvdw.
\end{align*}

To make the $L^2$ argument work in our analysis, we should obtain a favorable bound for the term inside the parenthesis on the right-hand side. This can be done by controlling derivatives of $\Phi(u,w) - \Phi(u,v)$. 
\begin{prop}\label{prop:phibound}
Let $(u_0,v_0)\in \mathcal A_{l,n}$. Assume that $|u-u_0|,|v-v_0|,|w-v_0|\le \varepsilon 2^{-k}$ for small $\varepsilon>0$. Then the following estimates hold.
\begin{enumerate}[topsep=-1pt, itemsep=2pt]
    \item [i)] Upper bounds of the phase difference. For $m\ge 1$,
    \begin{align}
\nonumber
\big|\partial_{u}^m(\Phi(u,v)-\Phi(u,w))\big|&\lesssim 2^{-(2\sigma-m)k}(1+|n|^{2\sigma-2})|v-w|.
\end{align}
    \item [ii)] Lower bounds of the phase difference. Assume that either $|n|\gg1$ or $\sigma$ is odd. Then
    \begin{equation}
\label{ineq:lophase1}
    \big|\partial_u( \Phi(u,v) -  \Phi(u,w))\big|\gtrsim 2^{-(2\sigma-1) k}(1+|n|^{2\sigma-2})|v-w|,
\end{equation}
Furthermore, if  $|n|\lesssim 1$ and $\sigma$ is even, then
\begin{equation}
\label{ineq:lophase2}
    2^k\big|\partial_u( \Phi(u,v) -  \Phi(u,w))\big|+\big|\partial_u^2( \Phi(u,v) -  \Phi(u,w))\big|\gtrsim 2^{-(2\sigma-2) k}|v-w|.
\end{equation}
\end{enumerate}
\end{prop}
Postponing the proof of the proposition until the next section and assuming its validity, we complete the proof of Proposition \ref{prop:mainred}.

\begin{proof}[Proof of Proposition \ref{prop:mainred}]
  Decomposing the support of $B_{k,l}^n(u,v)$ into finite small pieces, we may assume that $B_{k,l}^n(u,v)$ vanish unless $|u-u_0|\le \eps 2^{-k}$, $|v-v_0|\le \eps 2^{-k}$ with some $(u_0,v_0)\in \mathcal A_{l,n}$ and sufficiently small $\eps>0$. Fix such $u_0, v_0$.
  This implies that $B_{k,l}^n(u,v)B_{k,l}^n(u,w)$ is supported in $|v-w|\le 10\eps 2^{-k}$.
Now we define
\begin{align*}
    \wt B_{k,l}^n(u,v) &:= B_{k,l}^n(2^{-k}u + u_0,2^{-k}v + v_0), \\
    \wt \Phi(u,v) &:= \Phi(2^{-k}u + u_0,2^{-k}v + v_0), \\
    \wt{\mathcal T}_{k,l}^n f(u) &:= \int \wt B_{k,l}^n(u,v)e^{\frac ih \wt \Phi(u,v)} f(v) dv.
\end{align*}
\subsubsection*{When $|n|\gg 1$ or $\sigma$ is odd}

From \eqref{b:dbs} and Proposition \ref{prop:phibound}, it follows that, for $m\ge 1$ and $(u,v,w)\in \R^3$ such that $\wt B_{k,l}^n(u,v)\wt B_{k,l}^n(u,w)\neq 0$,
\begin{align}
    \begin{aligned}\label{e:scbdsc1}
        \big|\partial_u(\wt \Phi(u,v) - \wt \Phi(u,w))\big|&\gtrsim 2^{-(2\sigma+1)k}|n|^{2\sigma-2}|v-w|, \\
        \big|\partial_u^m(\wt \Phi(u,v) - \wt \Phi(u,w))\big|&\lesssim 2^{-(2\sigma+1)k}|n|^{2\sigma-2}|v-w|, \\
        \big|\partial_u^m \wt B_{k,l}^n(u,v)\big|&\lesssim 1.
    \end{aligned}
\end{align}
Also, $\wt B_{k,l}^n$ is supported in a rectangle whose side lengths are less than $10\eps$.
Thus, by van der Corput's lemma, we have
\begin{align*}
    \bigg|\int e^{\frac ih(\wt\Phi(u,w) - \wt\Phi(u,v))} \wt B_{k,l}^n(u,v) \wt B_{k,l}^n(u,w) du\bigg|\lesssim \big(1 + h^{-1}2^{-(2\sigma+1)k}|n|^{2\sigma-2}|v-w|\big)^{-1},
\end{align*}
which implies, by Young's inequality,
\begin{align}\nonumber
    \big\|\wt{\mathcal T}_{k,l}^n f\big\|_{L^2}^2&\lesssim \|f\|_{L^2}^2 \int_{-1}^1 \big(1 + h^{-1}2^{-(2\sigma+1)k}|n|^{2\sigma-2}|v|\big)^{-1} dv  \\\label{ineq:wtklnc1}
    &\lesssim h\,2^{(2\sigma+1)k}|n|^{-2(\sigma-1)} \|f\|_{L^2}^2.
\end{align}
To obtain an estimate for $\mathcal T_{k,l}^n$ from the above inequality, we make the change of variables $u \to 2^{-k}u + u_0$, $v \to 2^{-k}v + v_0$ to see $\|\mathcal T_{k,l}^n\|_{L^2\to L^2} = 2^{-k}\|\wt{\mathcal T}_{k,l}^n\|_{L^2\to L^2}$. Being combined with \eqref{ineq:wtklnc1}, it gives
\begin{align*}
    \big\|\mathcal T_{k,l}^n f\big\|_{L^2}\lesssim h^{\frac12} 2^{\frac{2\sigma-1}{2}k} |n|^{-(\sigma-1)} \|f\|_{L^2},
\end{align*}
where the bound is better than those in \eqref{e:ctnkl}.

\subsubsection*{When $|n|\lesssim 1$ and $\sigma$ is even} 
By \eqref{ineq:lophase2}, we have
\[
\sum_{j=1}^2\big|\partial_u^j( \wt\Phi(u,v) -  \wt\Phi(u,w))\big|\gtrsim 2^{-(2\sigma+1)k}|v-w|.
\]
Recall that the functions $\wt\Phi$, $\wt B_{k,l}^n$ satisfy the upper bounds in \eqref{e:scbdsc1}. Thus applying van der Corput's lemma gives
\begin{align}
    \begin{aligned}\nonumber
            &\bigg|\int e^{\frac ih(\wt\Phi(u,w) - \wt\Phi(u,v))} \wt B_{k,l}^n(u,v) \wt B_{k,l}^n(u,w) du\bigg| \lesssim \big|h^{-1}2^{-(2\sigma+1)k}|v-w|\big|^{-\frac12}.
    \end{aligned}
\end{align}
Therefore, by Young's inequality, we obtain
\[
\big\|\wt{\mathcal T}_{k,l}^n f\big\|_{L^2}^2 \lesssim h^\frac12\,2^{\frac{2\sigma+1}{2}k} \|f\|_{L^2}^2,
\]
which, combined with $\|\mathcal T_{k,l}^n\|_{L^2\to L^2} = 2^{-k}\|\wt{\mathcal T}_{k,l}^n\|_{L^2\to L^2}$, gives the second estimate in \eqref{e:ctnkl}. 
\end{proof}

\section{Proof of Proposition \ref{prop:phibound}}\label{sec:phase}

In this section, Proposition \ref{prop:phibound} is established. Recall that the length of the interval $I$ is assumed to be sufficiently small, thus for $u,v\in \mathcal A_{l,n}$, $|u|$, $|v|\le \eps$ with small $\eps>0$. It allows us to assume that $2^{-k}\le \eps$. Throughout this section, we assume $|u-u_0|, |v-v_0|, |w-v_0|\le \eps 2^{-k}$ for a fixed $(u_0,v_0)\in \mathcal A_{l,n}$.

\subsection{Expression for $\partial_u\partial_v \Phi$}
By the Taylor expansion, Proposition \ref{prop:phibound} follows once we get favorable bounds for the second-order derivative $\partial_u\partial_v \Phi$ and higher order derivatives of $\Phi$. For this purpose, we obtain an expression for $\partial_u\partial_v \Phi$, which indicates that $\partial_u\partial_v \Phi$ can be approximated by a polynomial.

Using \eqref{i:defPhi} and \eqref{i:seta0}, we compute
\begin{align}\nonumber
    \partial_u\partial_v \Phi(u,v) 
    &= \partial_v\Big(\partial_{(s,\eta)}\big(\phi(s,\gamma(u),\eta) - \inp{\gamma(v)}{\eta}\big)\big|_{(s,\eta) = (s(u,v), \eta(u,v))}\cdot \partial_u(s(u,v),\eta(u,v)) \\\nonumber
    & \qquad+ \inp{\dot\gamma(u)}{\partial_x\phi(s(u,v),\gamma(u),\eta(u,v))}\Big) \\\nonumber
    &=\inp{\dot\gamma(u)}{\partial_v k(u,v)},
\end{align}
 where 
\[
k(u,v) := \partial_x\phi(s(u,v),\gamma(u),\eta(u,v)).
\]
Recall that $p(\gamma(u), k(u,v)) = 0$. Differentiating both sides of the identity in $v$ gives $\partial_\xi p(\gamma(u), k(u,v))\partial_v k(u,v) = 0$. Using this along with $\dot\gamma(u) = \partial_\xi p(\gamma(u), \xi(u))$ yields
\[
\inp{\dot\gamma(u)}{\partial_v k(u,v)} = -\inp{\partial_\xi p(\gamma(u), k(u,v)) - \partial_\xi p(\gamma(u), \xi(u))}{\partial_v k(u,v)}.
\]
Thus, applying Taylor's theorem, we have
\begin{align}\label{i:2dphi1}
    \partial_u\partial_v\Phi(u,v) = (\xi(u) - k(u,v))^\intercal \partial_\xi^2 p(\gamma(u), \xi(u))\partial_v k(u,v) + R_1(u,v),
\end{align}
where $R_1(u,v)$ is defined by
\begin{align*}
    R_1(u,v) &:= -\frac 12\int_0^1 \big\langle \partial_\tau^2\big(\partial_\xi p(\gamma(u), \xi_{u,v}(\tau))\big), \partial_v k(u,v)\big\rangle (1-\tau) d\tau, \\
    \xi_{u,v}(\tau) &:= \tau k(u,v)+(1-\tau)\xi(u).
\end{align*}
Notice that $\partial_v k(u,v) = -\partial_v(\xi(u) - k(u,v))$. Hence one may appeal to the results \eqref{r:dseta}, \eqref{b:dseta} for the bounds for $\xi(u) - k(u,v)$ to derive certain bounds on $\partial_u\partial_v\Phi$. However, a straightforward calculation reveals that the resulting bound is insufficient for proving the proposition, particularly when $\sigma$ is large. Hence our main task will be to obtain more refined bounds on the size of the vector $\xi(u) - k(u,v)$ and its derivative in $v$. 


Our strategy is to find an appropriate polynomial in $u,v$ approximating $\xi(u) - k(u,v)$ well. For that, we start by seeking a relation between $\xi(u) - k(u,v)$ and $z_{v-u}(\gamma(u),\xi(u))-\gamma(v)$.
From \eqref{i:relphi}, we see that $z_{v-u}(\gamma(u),\xi(u))-\gamma(v)$ is equal to 
\begin{align}\label{i:relk}
    z_{v-u}(\gamma(u),\xi(u))-z_{-s(u,v)}(\gamma(u),k(u,v)).
\end{align}
By Taylor's theorem, \eqref{i:relk} is expressed by
\begin{align*}
    \partial_t z_t(&\gamma(u),\xi(u))\Big|_{t=v-u}(v-u+s(u,v))  + \partial_\xi z_{v-u}(\gamma(u),\xi(u)) (\xi(u)-k(u,v)) + R_2
\end{align*}
with the remainder term $R_2=R_2(u,v)$. Applying Taylor's theorem again in $t$  gives
\begin{align*}
     \quad\partial_t z_t(\gamma(u),&\,\xi(u))\big|_{t=v-u}(v-u+s(u,v)) \\
    & + \partial_\xi z_{v-u}(\gamma(u),\xi(u)) (\xi(u)-k(u,v)) \\
    =  \partial_\xi p(\gamma(u),&\,\xi(u)) (v-u+s(u,v)) \\
    &+ (v-u)\partial^2_\xi p(\gamma(u), \xi(u)) (\xi(u) - k(u,v)) + R_3(u,v),
\end{align*}
where $R_3$ denotes the remainder.
We have also used $\dot z_t(x,\xi) = \partial_\xi p(\kappa_t(x,\xi))$ and $\partial_\xi z_0(x,\xi) = 0$ for $x,\xi\in\R^2$.
Consequently, we deduce that
\begin{align}
\begin{aligned}\label{i:diff1}
    z_{v-u}(&\gamma(u),\xi(u))-\gamma(v) \\
    &= \partial_\xi p(\gamma(u),\,\xi(u)) (v-u+s(u,v)) + (v-u)\partial^2_\xi p(\gamma(u), \xi(u)) (\xi(u) - k(u,v))
\end{aligned}
\end{align}
modulo the remainder $R_2+R_3$. The following shows that the remainder terms $R_1, R_2, R_3$ are errors which we can discard in later argument.
\begin{lem}\label{lem:error}
    Let $\alpha\in \N_0^2$. Then we have
    \begin{align}\label{b:r1}
        \big|\partial_{(u,v)}^\alpha R_1(u,v)\big|\lesssim 2^{(|\alpha|-1)k}\big(|u|+2^{-k}\big)^{2\sigma-1},
    \end{align}
    and
    \begin{align}\label{b:r2r3}
        \big|\partial_{(u,v)}^\alpha R_2(u,v)\big| + \big|\partial_{(u,v)}^\alpha R_3(u,v)\big|\lesssim 2^{(|\alpha|-2)k}\big(|u|+2^{-k}\big)^{\sigma}.
    \end{align}
    Moreover, we have
    \begin{align}\label{b:errxik}
        \big|\partial_{(u,v)}^\alpha\big(\inp{\partial_\xi p(\gamma(u),\xi(u))}{\xi(u)-k(u,v)}\big)\big|\lesssim 2^{(|\alpha|-1)k}\big(|u|+2^{-k}\big)^{\sigma}.
    \end{align}
\end{lem}
Since the proof of this lemma is not straightforward, we postpone the proof to Section \ref{ssec:error} below. For the moment, we assume that the estimates of the lemma are valid. For simplicity, let $\mathcal E^{[1]} = \mathcal E^{[1]}(u,v)$, $\mathcal E^{[2]} = \mathcal E^{[2]}(u,v)$ denote ``errors", generic terms which may vary depending on the context, satisfying the bounds
\begin{align*}
    \big|\partial_{(u,v)}^\alpha \mathcal E^{[1]}(u,v)\big|&\lesssim 2^{(|\alpha|-1)k}\big(|u|+2^{-k}\big)^\sigma, \\
    \big|\partial_{(u,v)}^\alpha \mathcal E^{[2]}(u,v)\big|&\lesssim 2^{(|\alpha|-1)k}\big(|u|+2^{-k}\big)^{2\sigma-1}
\end{align*}
for $\alpha\in \N_0^2$. By Lemma \ref{lem:error}, $(v-u)^{-1}(R_2+R_3)=\mathcal E^{[1]}$ and $R_1=\mathcal E^{[2]}$.

Using \eqref{i:diff1} and the above lemma, we can express $\xi(u) - k(u,v)$ in terms of $z_{v-u}(\gamma(u),\xi(u))-\gamma(v)$. Note that, up to sign, there is a unique unit vector $\mathrm v = \mathrm v(u)\in \mathbb S^1$ such that $\inp{\mathrm v}{\partial_\xi p(\gamma(u),\xi(u))}=0$, $|\partial_u^m \mathrm v(u)|\le C_m$ with a constant $C_m>0$ for $m\in \N_0$.
By \eqref{b:errxik}, we get 
\begin{align}\label{ineq:diff3}
    \xi(u) - k(u,v) &= \inp{\xi(u) - k(u,v)}{\mathrm v}\mathrm v + \mathcal E^{[1]} =\frac{\langle \xi(u) - k(u,v),\partial_\xi^2 p\mathrm v \rangle}{\inp{\mathrm v}{\partial_\xi^2 p\mathrm v}}\mathrm v+ \mathcal E^{[1]}.
\end{align}
Here, $\partial_\xi^2 p=\partial_\xi^2 p(\gamma(u),\xi(u))$ and
we have used the facts that $|\partial_\xi p(\gamma(u),\xi(u))|$ and $|\inp{\mathrm v}{\partial_\xi^2 p(\gamma(u),\xi(u))\mathrm v}|$ are bounded below, which come from the conditions (A1) and (A2).
From \eqref{i:diff1} and \eqref{b:r2r3}, it follows that
\begin{align*}
    \langle \xi(u) - k(u,v),\partial_\xi^2 p\mathrm v \rangle
     = (v-u)^{-1} \inp{z_{v-u}(&\gamma(u),\xi(u))-\gamma(v)}{\mathrm v}+\mathcal E^{[1]}.
\end{align*}
Finally, combining the above inequality with \eqref{ineq:diff3}, we derive the expression
\begin{align}\label{ineq:diff4}
    \xi(u) - k(u,v)  =\frac{\langle z_{v-u}(\gamma(u),\xi(u))-\gamma(v),\mathrm v \rangle}{(v-u)\inp{\mathrm v}{\partial_\xi^2 p\mathrm v}}\mathrm v+ \mathcal E^{[1]}.
\end{align}

The advantage of utilizing $z_{v-u}(\gamma(u),\xi(u))-\gamma(v)$ is that, unlike $\xi(u) - k(u,v)$, the function $z_{v-u}(\gamma(u),\xi(u))-\gamma(v)$ is well-defined near the origin. Thus, considering the Taylor expansion at zero, we can approximate $z_{v-u}(\gamma(u),\xi(u))-\gamma(v)$ with a polynomial that remains invariant under the choice of parameters $k,n$. Indeed, using Taylor's theorem, we write
\begin{align*}
    z_{v-u}(\gamma(u),\xi(u))-\gamma(v) = &\sum_{j=0}^{\sigma+1}\Big(\partial_t^{j}z_t(\gamma(u),\xi(u))\Big|_{t=0} - \gamma^{(j)}(u)\Big)\frac{(v-u)^j}{j!}\\
    & + O((v-u)^{\sigma+2}).
\end{align*}
Recall that $\dot\gamma(u) = \partial_\xi p(\gamma(u), \xi(u))$ for every $u$. Thus the $0, 1$-th summands on the right-hand side vanish.
For the rest of the summands, we can verify that the following expressions hold.
\begin{lem}\label{lem:djbegam}
    Let $2\le j\le \sigma+1$ and
    \begin{align*}
    \mathfrak b:= \partial_\xi^2 p(\gamma(0),\xi(0))\Big(\xi^{(\sigma)}(0) + \partial_u^{\sigma-1}(\partial_x p(\gamma(u), \xi(u)))\Big|_{u=0}\Big).
\end{align*}
Then $|\inp{\mathfrak b}{\mathrm v}|\ge c$ with some constant $c>0$. Furthermore,
    \begin{align}\label{i:ztj}
        \partial_t^{j}&z_t(\gamma(u),\xi(u))\Big|_{t=0} - \gamma^{(j)}(u) = \mathfrak b \frac{u^{\sigma+1 - j}}{(\sigma+1-j)!} + O(u^{\sigma+2-j}).
    \end{align}
\end{lem}

For the same reason mentioned following Lemma \ref{lem:error}, we postpone the proof of Lemma \ref{lem:djbegam} until the end of the section and assume for the moment its validity.
From Lemma \ref{lem:djbegam}, it follows that
\begin{align}\label{i:diff2}
    z_{v-u}(\gamma(u),\xi(u))-\gamma(v) = P_\sigma(u,v) \mathfrak b + R_4(u,v),
\end{align}
where
\begin{align*}
    P_\sigma(u,v) := \sum_{j=2}^{\sigma+1} \frac{(v-u)^ju^{\sigma+1-j}}{j!(\sigma+1-j)!}
\end{align*}
and  $R_4$ denotes the remainder term expressed as
\begin{align*}
    R_4(u,v) = \sum_{\substack{m_1\ge 2 \\ \sigma+2\le m_1+m_2}} \frac{\mathfrak b_{m_1,m_2}}{m_1!\,m_2!} (u-v)^{m_1}u^{m_2},\quad \mathfrak b_{m_1,m_2}\in \R^2.
\end{align*}
From this, we see that $R_4$ satisfies the bound
\begin{equation}
    \label{ineq:r1}
\big|\partial_{(u,v)}^\alpha R_4(u,v)\big|\lesssim 2^{(|\alpha|-2)k}\big(|u|+2^{-k}\big)^{\sigma},\quad \alpha\in \N_0^2.
\end{equation}
On the other hand, by definition, it is clear that
\begin{align}\label{e:r1}
    \big|\partial_{(u,v)}^\alpha\big(P_\sigma(u,v)\mathfrak b\big)\big|\lesssim 2^{(|\alpha|-2)k}\big(|u|+2^{-k}\big)^{\sigma-1}
\end{align}
for $\alpha\in \N_0^2$.

We can now obtain the expressions we have sought to find in this subsection.
Combining \eqref{ineq:diff4} and \eqref{i:diff2} along with \eqref{ineq:r1}, we get
\begin{align}\label{i:diff5}
    \xi(u) - k(u,v) = (v-u)^{-1}P_\sigma(u,v)\inp{\mathfrak b}{\mathrm v}\big(\inp{\mathrm v}{\partial_\xi^2 p(\gamma(u), \xi(u))\mathrm v}\big)^{-1}\mathrm v + \mathcal E^{[1]}.
\end{align}
Since the derivatives of $\inp{\mathfrak b}{\mathrm v}\big(\inp{\mathrm v}{\partial_\xi^2 p(\gamma(u), \xi(u))\mathrm v}\big)^{-1}$ are uniformly bounded, substituting \eqref{i:diff5} into \eqref{i:2dphi1} and using \eqref{b:r1}, \eqref{e:r1} gives
\begin{align}\label{i:uvphiexp1}
    \partial_u\partial_v\Phi(u, v) = (v-u)^{-1}P_\sigma(u,v)\partial_v\big((v-u)^{-1}P_\sigma(u,v)\big)\mathcal A(u) + \mathcal E^{[2]},
\end{align}
where
\[
\mathcal A(u) = (\inp{\mathfrak b}{\mathrm v})^2\big(\inp{\mathrm v}{\partial_\xi^2 p(\gamma(u), \xi(u))\mathrm v}\big)^{-1}.
\]
Note that
\begin{align}\nonumber
    \big|\mathcal A(u)|\ge C,\quad \big|\partial_u^m \mathcal A(u)\big|\le C_m,\quad m\in \N_0.
\end{align}

Using the binomial theorem, we can express the terms $(v-u)^{-1}P_\sigma(u,v)$ and $\partial_v\big((v-u)^{-1}P_\sigma(u,v)\big)$ in a simpler manner. We first observe that
\[
P_\sigma(u,v) = \frac{1}{(\sigma+1)!}\big(v^{\sigma+1} - (\sigma+1)(v-u)u^\sigma - u^{\sigma+1}\big).
\]
This is an easy consequence of the binomial theorem. Using this, we calculate
\begin{align*}
    \partial_v\big((v-u)^{-1}P_\sigma(u,v)\big) = \frac{1}{(\sigma+1)!(v-u)^2}\big((\sigma+1)(v-u)v^\sigma - v^{\sigma+1} + u^{\sigma+1}\big).
\end{align*}
Now we set
\begin{align*}
    \wp_{\sigma,1}(\tau) &:= \frac{1}{(\sigma+1)!}\big((1+\tau)^{\sigma+1} - \tau^{\sigma+1} - (\sigma+1)\tau^\sigma\big), \\
    \wp_{\sigma,2}(\tau) &:= \frac{1}{(\sigma+1)!}\big(\tau^{\sigma+1} + (\sigma+1)(1+\tau)^\sigma - (1+\tau)^{\sigma+1}\big).
\end{align*}
Then from the above calculation, it follows that
\begin{align}
\begin{aligned}\nonumber
    (v-u)^{-1}P_\sigma(u,v) &= (v-u)^{\sigma}\wp_{\sigma,1}(u(v-u)^{-1}),\\ \partial_v\big((v-u)^{-1}P_\sigma(u,v)\big) &= (v-u)^{\sigma-1}\wp_{\sigma,2}(u(v-u)^{-1}).
\end{aligned}
\end{align}
Thus the expression \eqref{i:uvphiexp1} is rephrased as
\begin{align}\label{i:uvphiexp2}
    \partial_u\partial_v\Phi(u, v) = (v-u)^{2\sigma-1}\wp_{\sigma,1}(u(v-u)^{-1})\wp_{\sigma,2}(u(v-u)^{-1})\mathcal A(u) + \mathcal E^{[2]}.
\end{align}
Also, we observe that the polynomials $\wp_{\sigma,1}(\tau)$, $\wp_{\sigma,2}(\tau)$ satisfy the relation
\begin{align}\label{i:wprel}
    \wp_{\sigma,1}(\tau) = (-1)^{\sigma+1} \wp_{\sigma, 2}(-\tau-1),\quad  \wp_{\sigma, i}'(\tau) = \wp_{\sigma-1,i}(\tau)
\end{align}
for $i=1,2$.

\subsection{Proof of Proposition \ref{prop:phibound}} In this subsection, we prove Proposition \ref{prop:phibound} by assuming that Lemma \ref{lem:error}, \ref{lem:djbegam} are valid. 

We first note that \textit{i)} in Proposition \ref{prop:phibound} is an immediate consequence of the above discussion. Indeed, by the mean value theorem, we only need to show
\begin{align}\label{e:bdphase}
    \big|\partial_{(u,v)}^{\alpha}\partial_u\partial_v\Phi(u,v)\big|\lesssim 2^{(|\alpha|-1)k}\big(2^{-k} + 2^{-k}|n|\big)^{2\sigma-2},\quad \alpha\in \N_0^2.
\end{align}
However, the above estimate directly follows from \eqref{e:r1}, \eqref{i:uvphiexp1}, and $|u|\lesssim 2^{-k}|n|$.

To establish \textit{ii)} in Proposition \ref{prop:phibound}, we consider several cases separately.

\subsubsection*{When $|n|\gg 1$} We note that $\wp_{\sigma,1}$, $\wp_{\sigma,2}$ are polynomials of degree $\sigma -1$. Hence, if $|n|$ is sufficiently large,
\[
\big|\wp_{\sigma,1}(u(v-u)^{-1})\big|,\ \big|\wp_{\sigma,2}(u(v-u)^{-1})\big|\gtrsim \big|u(v-u)^{-1}\big|^{\sigma-1}
\]
because $|u(v-u)^{-1}|\gg 1$. Since $|v-u|\sim 2^{-k}$ and $|u|\sim |n|2^{-k}$, the above inequalities and \eqref{i:uvphiexp2} imply
\begin{align}\label{e:lbphiuvc1}
    \big|\partial_u\partial_v \Phi(u,v)\big|\gtrsim 2^{-(2\sigma-1)k}|n|^{2\sigma-2}.
\end{align}
We discard the error term $\mathcal E^{[2]}$ by using $|u|+2^{-k}\lesssim \eps$. 
Combining \eqref{e:bdphase}, \eqref{e:lbphiuvc1} together with Taylor's theorem gives
\begin{align}\nonumber
        \big|\partial_u(\Phi(u,w) - \Phi(u,v))\big|&\ge \big|\partial_u\partial_v\Phi(u,v)\big||v-w| - \Big(\sup_{(u,v)\in \mathcal A_{l,n}}\big|\partial_u\partial_v^2\Phi(u,v)\big|\Big)|v-w|^2 \\\label{e:lbphasec1}
        &\gtrsim 2^{-(2\sigma-1)k}|n|^{2\sigma-2}|v-w|.
\end{align}
In the second line, we also use that $|v-w|\lesssim \eps 2^{-k}$.

Now we suppose that $|n|\lesssim 1$. To handle the case, it is necessary to make use of the following observation.

\begin{lem}\label{lem:roots}
    Let $\sigma\ge 1$. For the polynomials $\wp_{\sigma,1}$, $\wp_{\sigma,2}$, we have
    \begin{enumerate}[topsep = -3pt, itemsep = 2pt]
        \item [(i)] If $\sigma$ is odd, then the polynomials $\wp_{\sigma,1}$, $\wp_{\sigma,2}$ have no roots on $\R$.
        \item [(ii)] If $\sigma$ is even, then the polynomials $\wp_{\sigma,1}$, $\wp_{\sigma,2}$ have exactly one root on $\R$. Furthermore, let $\tau_1$, $\tau_2$ be such a root of $\wp_{\sigma,1}$, $\wp_{\sigma,2}$, respectively. Then
        \begin{align}\label{e:t2<t1}
            -1 < \tau_2 < -1/2 < \tau_1 < 0.
        \end{align}
    \end{enumerate}
\end{lem}

\begin{proof}
    The lemma is a consequence of an elementary calculation. For $(i)$, by 
 \eqref{i:wprel}, it suffices to check that $\wp_{\sigma,1}(\tau)\ge C$ for every $\tau\in \R$ with a constant $C>0$.  We compute
    \begin{align*}
        \wp_{\sigma,1}'(\tau) &= \frac{1}{\sigma!}\big((1+\tau)^\sigma - \tau^\sigma - \sigma\tau^{\sigma-1}\big).
    \end{align*}
This polynomial is of degree $\sigma-2$, thus it has at least one root on $\R$ if $\sigma$ is odd. Let $\tau_1\in\R$ be arbitrary root of $\wp'_{\sigma,1}$. By a calculation, it follows that
\begin{align}\label{i:wp12}
    \wp_{\sigma,1}(\tau_1) = \sigma \big((\sigma+1)!\big)^{-1} \tau_1^{\sigma-1}.
\end{align}
Since $\wp_{\sigma,1}(0)=\big((\sigma+1)!\big)^{-1}>0$, $\tau_1\neq 0$. Thus $\wp_{\sigma,1}(\tau_1)>0$ if $\sigma-1$ is even. This means that there exists a constant $C>0$ such that $\wp_{\sigma,1}(\tau)\ge C$ for all $\tau\in \R$ if $\sigma$ is odd. Moreover, it also verifies the first statement of $(ii)$.
Thus it remains to show \eqref{e:t2<t1}. 
A straightforward calculation shows that 
\[
\wp_{\sigma,1}(-1/2) = -\sigma \big((\sigma+1)!\big)^{-1} (1/2)^\sigma <0,\quad \wp_{\sigma,1}(0) = \big((\sigma+1)!\big)^{-1}>0,
\]
yielding $-1/2 < \tau_1 < 0$. By \eqref{i:wprel}, it gives \eqref{e:t2<t1}.
\end{proof}

\subsubsection*{When $|n|\lesssim 1$ and $\sigma$ is odd}
By Lemma \ref{lem:roots},
\[
\big|\wp_{\sigma,1}(u(v-u)^{-1})\big|,\ \big|\wp_{\sigma,2}(u(v-u)^{-1})\big|\ge C
\]
with a constant $C>0$. Following the same argument as in the case $|n|\gg 1$, we deduce
\begin{align*}
    \begin{aligned}
        \big|\partial_u( \Phi(u,v) -  \Phi(u,w))\big|&\gtrsim 2^{-(2\sigma-1)k}|v-w|.
    \end{aligned}
\end{align*}
Combined with \eqref{e:lbphasec1}, it completes the proof of \eqref{ineq:lophase1}.

\begin{figure}
    \centering
    \begin{tikzpicture}[use Hobby shortcut, scale=0.8] 
\draw (-3,2) .. (-1.2,-0.2) .. (0,-0.5) .. (1.2,-0.2) .. (3,2); 
\draw (-3.5,0.3) .. (-1.2,-0.2) .. (1.5,-2.5);
\node[scale=1, left] at (-3,2) {$\gamma$};
\node[scale=1, above] at (-3.5,0.3) {$z_{v-u}$};
\node[scale=0.9, above] at (-0.8,-0.3) {$\gamma(u)$};
\end{tikzpicture}
\qquad
\begin{tikzpicture}[use Hobby shortcut, scale=0.8] 
\draw (-3,2) .. (-1.2,-0.2) .. (0,-0.5) .. (1.2,-0.8) .. (3,-3); 
\draw (-3.5,0) .. (-1.2,-0.23) .. (3.3,-3);
\node[scale=1, left] at (-3,2) {$\gamma$};
\node[scale=1, above] at (-3.5,0) {$z_{v-u}$};
\node[scale=0.9, above] at (-0.8,-0.3) {$\gamma(u)$};
\end{tikzpicture}
\caption{The graphs of $v\mapsto \gamma(v)$ and $z_{v-u}(\gamma(u), \xi(u))$ when $\sigma$ is odd (left) and even (right). Note that $z_{v-u}(\gamma(u), \xi(u))$ intersects with $\gamma$ at a point other than $\gamma(u)$ on the right figure.}
\label{fig:casesigma}
\end{figure}
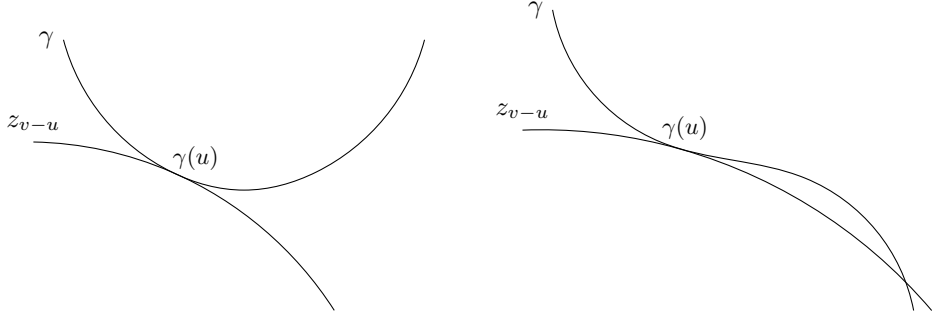

\subsubsection*{When $|n|\lesssim 1$ and $\sigma$ is even} In this case, Lemma \ref{lem:roots} tells us that either $\wp_{\sigma,1}(u(v-u)^{-1})$ or $\wp_{\sigma,2}(u(v-u)^{-1})$ may become zero. It implies that $\partial_u\partial_v \Phi(u,v)$ could be zero at some points $(u,v)\in \supp(B_{k,l}^n)$. To understand why this phenomenon occurs, we recall \eqref{i:uvphiexp1} and \eqref{i:diff2}. These expressions suggest that whether $\partial_u\partial_v \Phi(u,v)$ vanishes on $\supp(B_{k,l}^n)$ comes from whether $z_{v-u}(\gamma(u),\xi(u))$ intersects with $\gamma(v)$ at some $(u,v)\in \supp(B_{k,l}^n)$, if we disregard the contribution from the error. The intersection of the two curves $z_{v-u}(\gamma(u),\xi(u))$ and $\gamma(v)$ is illustrated in Figure \ref{fig:casesigma}. As clear from Figure \ref{fig:casesigma}, these two curves have two intersections if $\sigma$ is even.

However, one can show that the third-order derivative $\partial_u^2\partial_v\Phi(u,v)$ does not become zero around the zero set of $\partial_u\partial_v\Phi(u,v)$. This can be seen by observing that the functions $\wp_{\sigma,i}(u(v-u)^{-1})$, $\partial_u\wp_{\sigma,i}(u(v-u)^{-1})$ does not simultaneously vanish at the same point for $i = 1,2$. Indeed, let us set
\begin{align*}
    \wt{\wp}_{\sigma,i}(\tau) := (1+\tau) \wp_{\sigma-1,i}(\tau),\quad i=1,2.
\end{align*}
Since $\wp_{\sigma,i}'(\tau) = \wp_{\sigma-1,i}(\tau)$, by a calculation we have
\begin{align*}
    \partial_u\big(\wp_{\sigma,i}(u(v-u)^{-1})\big) = (v-u)^{-1}\wt{\wp}_{\sigma,i}(u(v-u)^{-1}),\quad i=1,2.
\end{align*}
Lemma \ref{lem:roots} implies that both polynomials $\wt{\wp}_{\sigma,1}(\tau)$, $\wt{\wp}_{\sigma,2}(\tau)$ vanish only when $\tau = -1$. Thus, by \eqref{e:t2<t1}, $\wt{\wp}_{\sigma,1}(\tau_1)$, $\wt{\wp}_{\sigma,2}(\tau_2)\neq 0$.
This justifies the above observation.

Therefore, it follows that
\begin{align*}
    \big|\wp_{\sigma,1}(u(v-u)^{-1})&\wp_{\sigma,2}(u(v-u)^{-1})\big| \\
    &+ \big|(v-u)\partial_u\big(\wp_{\sigma,1}(u(v-u)^{-1})\wp_{\sigma,2}(u(v-u)^{-1})\big)\big|\ge C
\end{align*}
for a constant $C>0$. By the chain rule, the above bound implies
\begin{align*}
    \big|(v&-u)^{2\sigma-1}\wp_{\sigma,1}(u(v-u)^{-1})\wp_{\sigma,2}(u(v-u)^{-1})\mathcal A(u)\big| \\
    &+ 2^{-k}\big|\partial_u\big((v-u)^{2\sigma-1}\wp_{\sigma,1}(u(v-u)^{-1})\wp_{\sigma,2}(u(v-u)^{-1})\mathcal A(u)\big)\big|\ge C2^{-(2\sigma-1)k}.
\end{align*}
Combining with \eqref{i:uvphiexp2} shows \eqref{ineq:lophase2}. This completes the proof of \textit{ii)} in Proposition \ref{prop:phibound}.

\subsection{Estimates for error terms}\label{ssec:error}
In this subsection, we prove Lemma \ref{lem:error}. 
We start with obtaining the required bounds for $R_2$, $R_3$.
    By the Taylor theorem, these terms are expressed by
\begin{align*}
    R_2(u,v) & = \sum_{\mathfrak m\in \N_0^3 : |\mathfrak m|=2} \mathcal R_{2,\mathfrak m}(u,v) (v-u+s(u,v), \xi(u)-k(u,v))^\mathfrak m, \\
    R_3(u,v) & = \mathcal R_{3,1}(u,v) \big(v-u+s(u,v)\big) (v-u) + \mathcal R_{3,2}(u,v)\big(\xi(u) - k(u,v)\big)(v-u)^2,
\end{align*}
where
\begin{align*}
    \mathcal R_{2,\mathfrak m}(u,v) &:= -\frac 12\int_0^1 \partial_{(t,\xi)}^\mathfrak m z_t(\gamma(u), \xi)\bigg|_{t= t_{u,v}(\tau), \xi = \xi_{u,v}(\tau)} (1-\tau) d\tau, \\
    \mathcal R_{3,1}(u,v) &:= \int_0^1 \partial_t^2 z_t(\gamma(u), \xi(u))\bigg|_{t= \tau (v-u)} d\tau, \\
    \mathcal R_{3,2}(u,v) &:= \frac 12\int_0^1 \partial_t^2 \partial_\xi z_t(\gamma(u), \xi(u))\bigg|_{t= \tau (v-u)} (1-\tau) d\tau,
\end{align*}
and $t_{u,v}(\tau) := -\tau s(u,v)+(1-\tau)(v-u)$. Also, recall that $\xi_{u,v}(\tau) = \tau k(u,v)+(1-\tau)\xi(u)$.
Since derivatives of $z_t(\gamma(u), \xi(u))$ in $t, u$ are uniformly bounded, we have
\begin{align}\label{e:mcr3}
    \big|\partial_{(u,v)}^\beta \mathcal R_{3,1}(u,v)\big|, \ \big|\partial_{(u,v)}^\beta \mathcal R_{3,2}(u,v)\big|\le C_\beta
\end{align}
with a constant $C_\beta>0$ for $\beta\in \N_0^2$.
Also, using the bounds for the derivatives of $s(u,v)$, $\eta(u,v)$ in Lemma \ref{lem:seta}, we obtain
\begin{align}\label{e:bdonk}
    \big|\partial_{(u,v)}^\beta k(u,v)\big|\le C_\beta 2^{|\beta|k}, \quad \beta\in \N_0^2,
\end{align}
which implies that, for every $\beta\in\N_0^2$,
\begin{align}\label{e:bdr21}
    \big|\partial_{(u,v)}^\beta \mathcal R_{2,\mathfrak m} (u,v)\big|\le C_\alpha 2^{|\beta|k}, \quad \mathfrak m\in \N_0^3\ \ \text{such that}\ \  |\mathfrak m|=2,\ \mathfrak m_1\ge 1.
\end{align}
If $\mathfrak m_1 = 0$, $\partial_{(t,\xi)}^\mathfrak m z_t(\gamma(u),\xi) = \partial_{\xi}^{\mathfrak m'} z_t(\gamma(u),\xi)$ where $\mathfrak m = (\mathfrak m_1, \mathfrak m')$. From the Taylor theorem, it follows that
\begin{align*}
    \partial_{\xi}^{\mathfrak m'} z_{t_{u,v}(\tau)}(\gamma(u),\xi)\Big|_{\xi = \xi_{u,v}(\tau)} =  \int_0^{t_{u,v}(\tau)} \partial_\xi^{\mathfrak m'}\big(\partial_\xi p(\kappa_{\tilde\tau}(\gamma(u), \xi))\big)\Big|_{\xi = \xi_{u,v}(\tau)} d\tilde\tau.
\end{align*}
We have also used $\partial_t z_t = \partial_\xi p(\kappa_t)$. Since $|t_{u,v}(\tau)|\lesssim 2^{-k}$, this implies that, for $(u,v)\in \mathcal A_{l,n}$ and $\beta\in\N_0^2$,
\begin{align}\label{e:bdr22}
    \big|\partial_{(u,v)}^\beta \mathcal R_{2,\mathfrak m} (u,v)\big|\le C_\beta 2^{(|\beta|-1)k}, \quad \mathfrak m\in \N_0^3\ \ \text{such that}\ \  |\mathfrak m|=2,\ \mathfrak m_1 = 0.
\end{align}

By \eqref{e:mcr3}--\eqref{e:bdr22}, the proof comes down to estimating the derivatives of $v-u-s(u,v)$ and $\xi(u) - k(u,v)$. As mentioned earlier, utilizing \eqref{r:dseta} directly results in obtaining crude estimates compared to the desired result. It turns out that one can refine \eqref{r:dseta} to a significant extent if $\sigma$ is large. Let
\[
\mathfrak d^\alpha = \mathfrak d^\alpha(u,v):=\big|\partial_{(u,v)}^\alpha\big(v-u+s(u,v)\big)\big| + 2^{-k} \big|\partial_{(u,v)}^\alpha\big(\xi(u) - k(u,v)\big)\big|.
\]

\begin{lem}\label{lem:setab}
    Let $(u,v)\in \mathcal A_{l,n}$. Then for every $\alpha\in \N_0^2$, we have
    \begin{align}
        \begin{aligned}\label{e:setab}
            \mathfrak d^\alpha(u,v)&\le C_\alpha 2^{(|\alpha|-2)k}\big(|u|+2^{-k}\big)^{\sigma-1}.
        \end{aligned}
    \end{align}
\end{lem}

\begin{proof}

The proof is by induction on $|\alpha|$. We first work with the case $\alpha=0$. From \eqref{i:1dphix}, it follows that $|k(u,v)-\eta(u,v)|\lesssim 2^{-k}$. By \eqref{r:dseta}, this entails $|\xi(u)-k(u,v)|\lesssim 2^{-k}$. Combining with the bounds \eqref{e:mcr3}--\eqref{e:bdr22}, we get
\begin{align}\label{b:r2pr3}
    |R_2(u,v)|+ |R_3(u,v)|\lesssim 2^{-k}\mathfrak d^0(u,v).
\end{align}
Also, using \eqref{i:diff1} and \eqref{i:diff2}, we see that
\begin{align}
\begin{aligned}\label{i:diff3}
    &\partial_\xi p(\gamma(u), \xi(u)) (v-u+s(u,v)) + (v-u)\partial_\xi^2 p(\gamma(u), \xi(u)) (\xi(u) - k(u,v)) \\
    &\qquad = P_\sigma(u,v)\mathfrak b + \big(R_4-R_2-R_3\big)(u,v).
\end{aligned}
\end{align}
Combining with \eqref{ineq:r1}, \eqref{e:r1}, \eqref{b:r2pr3}, we have
\begin{equation}
\label{ineq:alpha=0}
\begin{aligned}
|\partial_\xi p(\gamma(u),\,&\xi(u)) (v-u+s(u,v)) + (v-u)\partial_\xi^2 p (\gamma(u), \xi(u)) (\xi(u) - k(u,v))| \\
&\lesssim    2^{-2k}(|u|+2^{-k})^{\sigma-1}+2^{-k}\mathfrak d^0(u,v).
\end{aligned}    
\end{equation}

If $\xi(u) - k(u,v) = 0$, then the desired result directly follows from \eqref{ineq:alpha=0} and that $|\partial_\xi p(\gamma(u),\xi(u))|\ge C$. Thus we may assume $\xi(u) - k(u,v) \neq 0$. Then we can observe that $\xi(u)-k(u,v)$ and $\partial_\xi p(\gamma(u), \xi(u))$ are almost orthogonal. Indeed,
recall that $p(\gamma(u), k(u,v)) = p(\gamma(u), \xi(u)) = 0$. Thus applying Taylor's theorem gives
\begin{align}
\begin{aligned}\label{ni:dpkxi}
    0 &= p(\gamma(u), \xi(u)) - p(\gamma(u), k(u,v)) \\
    &= \inp{\partial_\xi p(\gamma(u), \xi(u))}{\xi(u)-k(u,v)} + O(|\xi(u)-k(u,v)|^2),
\end{aligned}
\end{align}
implying 
\begin{align}\nonumber
    |\inp{\partial_\xi p(\gamma(u), \xi(u))}{\xi(u)-k(u,v)}|\lesssim 2^{-k}|\xi(u)-k(u,v)|.
\end{align}
By the assumption (A2), this says that the angle between the two vectors
\[
\partial_\xi p(\gamma(u), \xi(u)),\quad \partial_\xi^2 p(\gamma(u), \xi(u))(\xi(u)-k(u,v))
\]
are separated from $0,\pi$ by some constant $c>0$.
Note that for $\mathrm v_1,\mathrm v_2\in \R^2$, the inequality
\begin{align}\label{e:perpineq}
    |\mathrm v_1+\mathrm v_2|\ge \Big(\frac{1+\cos(\angle(\mathrm v_1,\mathrm v_2))}{2}\Big)^{\frac12}(|\mathrm v_1|+|\mathrm v_2|)
\end{align}
holds. Thus, the left-hand side of \eqref{ineq:alpha=0} is bounded below by a constant times $\mathfrak d^0$, which implies that
\[
\mathfrak d^0(u,v)\lesssim2^{-2k}(|u| + 2^{-k})^{\sigma-1}.
\]

Now suppose that for a fixed $j\ge 1$, the assertion holds for every $\alpha\in \N_0^2$ such that $|\alpha|<j$. Let us set $|\alpha|=j$. By the induction hypothesis and the chain rule, differentiating \eqref{i:diff3} yields
\begin{align*}
\begin{aligned}
    &|\partial_\xi p \partial_{(u,v)}^\alpha(v-u+s(u,v)) + (v-u)\partial_\xi^2 p \partial_{(u,v)}^\alpha(\xi(u) - k(u,v))| \\
    &\qquad \lesssim  |\partial_{(u,v)}^\alpha P_\sigma(u,v)\mathfrak b + \partial_{(u,v)}^\alpha\big(R_4 - R_2 - R_3\big)(u,v)|+2^{(|\alpha|-2)k}(|u|+2^{-k})^{\sigma-1}.
\end{aligned}
\end{align*}
We have also used that the derivatives of $\partial_\xi p$, $\partial_\xi^2 p$ are uniformly bounded.
Also, the bounds \eqref{e:mcr3}--\eqref{e:bdr22} and the induction hypothesis imply that
\begin{align}
\begin{aligned}\label{e:devr2r3}
    |\partial_{(u,v)}^\alpha& R_2(u,v)| + |\partial_{(u,v)}^\alpha R_3(u,v)| \lesssim 2^{(|\alpha|-3)k}(|u|+2^{-k})^{\sigma-1}+2^{-k}\mathfrak d^\alpha.  
\end{aligned}
\end{align}
By the above two inequalities and the estimates \eqref{ineq:r1}, \eqref{e:r1}, we have
\begin{equation}
\label{ineq:alpha>0}
    \begin{aligned}
        &\big|\partial_\xi p \partial_{(u,v)}^\alpha(v-u+s(u,v)) + (v-u)\partial_\xi^2 p \partial_{(u,v)}^\alpha(\xi(u) - k(u,v))\big| \\
        &\quad\lesssim 2^{(|\alpha|-2)k}(|u|+2^{-k})^{\sigma-1}+2^{-k}\mathfrak d^\alpha.
    \end{aligned}
\end{equation}
The term in the second line of \eqref{ni:dpkxi} can be explicitly expressed as
\begin{align}\label{i:dpkxi}
    0 = \inp{\partial_\xi p(\gamma(u), \xi(u))}{\xi(u)-k(u,v)} + R_5(u,v),
\end{align}
where
\begin{align}\nonumber
    R_5(u,v) &:= \big(\xi(u)-k(u,v)\big)^\intercal\mathcal R_5(u,v) \big(\xi(u)-k(u,v)\big), \\\nonumber
    \mathcal R_5(u,v) &:=-\frac 12 \int_0^1 \partial_\xi^2 p(\gamma(u), \xi_{u,v}(\tau)) (1-\tau) d\tau.
\end{align}
By \eqref{e:bdonk}, $\mathcal R_5$ satisfies
\begin{align}\nonumber
    \big|\partial_{(u,v)}^\alpha \mathcal R_5(u,v)\big|\le C_\alpha 2^{|\alpha|k}
\end{align}
for $\alpha\in \N_0^2$.
Taking the derivative $\partial_{(u,v)}^\alpha$ on both sides of \eqref{i:dpkxi} and utilizing the induction hypothesis, we obtain
\begin{align}
\begin{aligned}\label{ne:hdinp1}
    \big|\big\langle &\partial_{(u,v)}^{\alpha}(\xi(u)-k(u,v)), \partial_\xi p(\gamma(u), \xi(u))\big\rangle\big| \\
    &\quad\lesssim 2^{(|\alpha|-2) k}(|u|+2^{-k})^{\sigma - 1} + 2^{-k}\big|\partial_{(u,v)}^{\alpha}(\xi(u)-k(u,v))\big|.
\end{aligned}
\end{align}
Let $\mathrm v_{\alpha}(u,v) := \langle\partial_{(u,v)}^{\alpha}(k(u,v) - \xi(u)), \mathrm v(u)\rangle \mathrm v(u)$.
Then by \eqref{ineq:alpha>0}, \eqref{ne:hdinp1}, we get
\begin{align*}
\begin{aligned}
        &\big|\partial_\xi p(\gamma(u), \xi(u)) \partial_{(u,v)}^\alpha(v-u+s(u,v)) + (v-u)\partial_\xi^2 p(\gamma(u), \xi(u))\mathrm v_\alpha(u,v)\big| \\
        &\quad\lesssim 2^{(|\alpha|-2)k}(|u|+2^{-k})^{\sigma-1}+2^{-k}\mathfrak d^\alpha.
\end{aligned}
\end{align*}
Since $\mathrm v_\alpha$ is perpendicular to the vector $\partial_\xi p(\gamma(u),\xi(u))$, applying \eqref{e:perpineq} to the above shows
\begin{align*}
        &\big|\partial_{(u,v)}^\alpha(v-u+s(u,v))\big| + 2^{-k}\big|\inp{\partial_{(u,v)}^{\alpha}(\xi(u)-k(u,v))}{\mathrm v}\big| \\
        &\qquad\quad\lesssim 2^{(|\alpha|-2)k}(|u|+2^{-k})^{\sigma-1}+2^{-k}\mathfrak d^\alpha.
\end{align*}
From the estimate and \eqref{ne:hdinp1}, the desired assertion \eqref{e:setab} follows.
\end{proof}
We now prove Lemma \ref{lem:error}. 
\begin{proof}[Proof of Lemma \ref{lem:error}]
We first note that the bounds \eqref{b:r2r3} and \eqref{b:errxik} are direct results of \eqref{e:devr2r3}, \eqref{ne:hdinp1}, and Lemma \ref{lem:setab}. To establish the bound \eqref{b:r1}, we observe that $2R_1=\inp{(R_{1,1}(u,v),R_{1,2}(u,v))}{\partial_vk(u,v)}$, where
    \begin{align*}
        R_{1,i}(u,v) := -\int_0^1 (\xi(u)-k(u,v))^\intercal\partial_\xi^2\partial_{\xi_i} p(\gamma(u),\xi_{u,v}(\tau))(\xi(u)-k(u,v))(1-\tau) d\tau
    \end{align*}
    for $i = 1,2$.
    Since each entry of the matrix $\partial_{(u,v)}^\alpha\big(\partial_\xi^2\partial_{\xi_i} p(\gamma(u),\xi_{u,v}(\tau))\big)$ is bounded by $C2^{|\alpha|k}$, using Lemma \ref{lem:setab} together with the chain rule, we obtain \eqref{b:r1}, as desired.
\end{proof}

\subsection{Proof of Lemma \ref{lem:djbegam}}
We first prove that the expression \eqref{i:ztj} is valid.
We claim that, for $0\le m\le \sigma-1$,
    \begin{align}\label{i:dmsig}
        \partial_t^m\zeta_t(\gamma(0),\xi(0))\Big|_{t=0} = \xi^{(m)}(0)
    \end{align}
    and
    \begin{align}
    \begin{aligned}\label{i:dabet}
        \partial_t^{\sigma+1}&z_t(\gamma(0), \xi(0))\Big|_{t=0} - \gamma^{(\sigma+1)}(0) = \mathfrak b.            
    \end{aligned}
    \end{align}

    Assume for the moment that the claim is true.
    When $j=\sigma+1$, the result immediately follows from the Taylor theorem. Now consider the case $2\le j\le \sigma$. Applying Taylor's theorem, for $j\ge2$,
\begin{align}
\begin{aligned}\label{i:djbeta}
    \partial_t^{j}&z_t(\gamma(u),\xi(u))\Big|_{t=0} - \gamma^{(j)}(u) \\
        &= \sum_{\nu=0}^{\sigma+1-j}\bigg(\partial_u^\nu\Big(\partial_t^{j}z_t(\gamma(u),\xi(u))\Big|_{t=0}\Big)\bigg|_{u=0}-\gamma^{(j+\nu)}(0)\bigg)\frac{u^\nu}{\nu!} + O(u^{\sigma+2-j}).
\end{aligned}
\end{align}
To proceed, we make use of the following observation: for $j\in\N_0$ and $\nu\le \sigma-1$,
\begin{align}
    \begin{aligned}\label{i:mdbsi}
        \partial_u^\nu\Big(\partial_t^{j}z_t(\gamma(u),\xi(u))\Big|_{t=0}\Big)\Big|_{u=0} &= \partial_t^{j+\nu}z_t(\gamma(0), \xi(0))\Big|_{t=0}.
    \end{aligned}
\end{align}
To see this, set $\mathcal Q_0(z,\zeta) := z$ and 
\[
\mathcal Q_{j+1}(z,\zeta) := \partial_z\mathcal Q_j(z,\zeta)\partial_\xi p(z,\zeta) - \partial_\zeta\mathcal Q_j(z,\zeta)\partial_x p(z,\zeta).
\]
Then
\begin{align*}
    \partial_t^j z_t(\gamma(u), \xi(u)) = \mathcal Q_j(\kappa_t(\gamma(u), \xi(u))).
\end{align*}
One can prove the above identity by using the induction on $j$. Hence
it suffices to show
\[
\partial_u^\nu\big(\mathcal Q_j(\gamma(u), \xi(u))\big)\big|_{u=0} = \partial_t^\nu\big(\mathcal Q_j(z_t(\gamma(0), \xi(0)), \zeta_t(\gamma(0), \xi(0))\big)\big|_{t=0}.
\]
This can be done by a routine calculation using the chain rule. One also needs to use \eqref{i:dmsig} and the assumption
\begin{align}\label{i:mdgamma}
    \partial_t^m z_t(\gamma(0), \xi(0))\Big|_{t=0} = \gamma^{(m)}(0),\quad 0\le m\le \sigma
\end{align}
to deduce the above formula.

Substituting \eqref{i:mdbsi} into \eqref{i:djbeta} and using \eqref{i:mdgamma} yields
\begin{align*}
    \partial_t^{j}z_t(\gamma(u),&\xi(u))\Big|_{t=0} - \gamma^{(j)}(u) \\
    &= \Big(\partial_t^{\sigma+1}z_t(\gamma(0), \xi(0))\Big|_{t=0} - \gamma^{(\sigma+1)}(0)\Big)\frac{u^{\sigma+1-j}}{(\sigma+1-j)!} + O(u^{\sigma+2-j}).
\end{align*}
Applying \eqref{i:dabet} to the above proves the required assertion.

Now we prove the claim. The proof of \eqref{i:dmsig} is by induction on $m$. The case $m=0$ is trivial, so assume that \eqref{i:dmsig} is true for $0\le m\le m'$, where $0\le m'\le \sigma-2$. To prove the statement with $m=m'+1$, we first observe that
\[
\partial_t^{\nu+1}\zeta_t(\gamma(0), \xi(0))\Big|_{t=0} = -\partial_u^\nu\Big(\partial_x p(\gamma(u),\xi(u))\Big)\Big|_{u=0}, \quad 0\le \nu\le m'.
\]
To see this, note that the identity is equivalent to
\begin{align*}
    \partial_t^{\nu}\big(\partial_x p(z_t(\gamma(0), \xi(0)), \zeta_t(\gamma(0), \xi(0)))\big)\big|_{t=0} = \partial_u^\nu\big(\partial_x p(\gamma(u),\xi(u))\big)\big|_{u=0}
\end{align*}
for $0\le \nu\le m'$. This follows from a direct calculation utilizing the induction hypothesis and \eqref{i:mdgamma}. Hence it is sufficient to show
\begin{align}\label{i:xim'}
    - \partial_u^{m'}\Big(\partial_x p(\gamma(u),\xi(u))\Big)\Big|_{u=0} = \xi^{(m'+1)}(0).
\end{align}
We recall the identities \eqref{i:perp1}--\eqref{i:gt3}. Using the same argument as in the proof of Lemma \ref{lem:finite}, we conclude $\dot\xi(0) + \partial_x p(\gamma(0),\xi(0)) = 0$. Hence \eqref{i:xim'} with $m'=0$ holds.

Next, suppose that $m'\ge 1$. Then we note that our induction hypothesis implies
\begin{align}\label{i:hyxinu}
    - \partial_u^{\nu}\Big(\partial_x p(\gamma(u),\xi(u))\Big)\Big|_{u=0} = \xi^{(\nu+1)}(0),\quad 0\le \nu\le m'-1.
\end{align}
Differentiating both sides of $p(\gamma(u),\xi(u)) = 0$ $m'+1$ times gives
\begin{align*}
    \sum_{\nu = 0}^{m'} C_\nu \Big\langle \partial_u^\nu\big(\partial_\xi p(\gamma(u), \xi(u))\big), \xi^{(m'-\nu+1)}(u) + \partial_u^{m'-\nu}\big(\partial_x p(\gamma(u),\xi(u))\big) \Big\rangle = 0,
\end{align*}
where $C_\nu$, $0\le \nu\le m'$, is a positive constant. 
Taking $u=0$ and then using the hypothesis \eqref{i:hyxinu} yields
\begin{align}\label{i:dmxiperp}
    \big\langle\partial_\xi p(\gamma(0), \xi(0)), \xi^{(m'+1)}(0) + \partial_u^{m'}(\partial_x p(\gamma(u), \xi(u)))\big|_{u=0}\big\rangle = 0.
\end{align}
We now observe that, by differentiating $\dot\gamma(u) = \partial_\xi p(\gamma(u),\xi(u))$ $m'+1$ times, one can express $\gamma^{(m'+2)}(0)$ as
\begin{align}\label{i:gamm'p}
    \gamma^{(m'+2)}(0) = \mathcal P_{m'}(\dot\gamma(0),\dots,\gamma^{(m'+1)}(0),\dot\xi(0),\dots,\xi^{(m')}(0)) + \partial_\xi^2 p(\gamma(0),\xi(0))\xi^{(m'+1)}(0),
\end{align}
where $\mathcal P_{m'}$ denotes a multivariate polynomial of degree $m'+1$. One can deduce this by direct computation. By the same argument, we obtain
\begin{align}
\begin{aligned}\label{i:betam'p}
    \partial_t^{m'+2}z_t(\gamma(0)&,\xi(0))\big|_{t=0} \\
    = \mathcal P_{m'}&\Big(\Big(\partial_t^\nu z_t(\gamma(0),\xi(0))\big|_{t=0}\Big)_{\nu=1}^{m'+1}, \Big(\partial_t^{\nu'} \zeta_t(\gamma(0),\xi(0))\big|_{t=0}\Big)_{\nu'=1}^{m'}\Big) \\
    &\qquad + \partial_\xi^2 p(\gamma(0),\xi(0))\cdot \partial_t^{m'+1}\zeta_t(\gamma(0),\xi(0))\big|_{t=0},
\end{aligned}
\end{align}
where $(\partial_t^\nu z_t(\gamma(0),\xi(0))|_{t=0})_{\nu=1}^{m'+1}$, $(\partial_t^{\nu'} z_t(\gamma(0),\xi(0))|_{t=0})_{\nu'=1}^{m'}$ denote the vectors in $\R^{2(m'+1)}$, $\R^{2m'}$, respectively.
Making use of \eqref{i:mdgamma} and the induction hypothesis, we get
\begin{align*}
    \partial_\xi^2 p(\gamma(0),\xi(0))\Big(\xi^{(m'+1)}(0) + \partial_u^{m'}(\partial_x p(\gamma(u), \xi(u)))\big|_{u=0}\Big) = 0.
\end{align*}
Following the same argument as what we have used in the case $m'=0$, we see that the above identity and \eqref{i:dmxiperp} implies that $\xi^{(m'+1)}(0) + \partial_u^{m'}(\partial_x p(\gamma(u), \xi(u)))\big|_{u=0} = 0$. This is the desired conclusion completing the proof of \eqref{i:dmsig}. \eqref{i:dabet} is a direct consequence of the above computation. Indeed, taking $m'=\sigma-1$ into \eqref{i:gamm'p}, \eqref{i:betam'p} and subtracting these two identities, we obtain \eqref{i:dabet} as an application of \eqref{i:dmsig} and \eqref{i:mdgamma}. 

It remains to show that $|\inp{\mathfrak b}{\mathrm v}|>0$. However, this is an immediate consequence of the definition of $\mathfrak b$ stated in the lemma and the identity \eqref{i:dmxiperp} with $m'=\sigma-1$, and the condition (A2). Therefore, all the assertions of the lemma follow.

\section{On the optimality of quasimode and eigenfunction estimates}\label{sec:optimal}

\subsection{Optimality of Theorem \ref{thm:main}}
In this subsection, we establish the sharpness of Theorem \ref{thm:main}. The strategy is to deduce it as a consequence of the following statement.

\begin{prop}\label{prop:sharp}
    Let $p, \gamma, h_0$ be as in Theorem \ref{thm:main}. Let $\varphi\in C_c^\infty(\R)$ such that $\supp(\varphi)\subset(-2\eps_0,2\eps_0)$, $\varphi \equiv 1$ on $[-\eps_0,\eps_0]$ for a sufficiently small $\eps_0>0$. Assume that $\sigma_{\gamma, p}\ge 1$. Then for every $\eps > 0$, there exist $\chi\in C_c^\infty(T^*M)$ and a collection of $L^2$ normalized functions $f = f(h)$, $h\in (0,h_0]$, such that for $h\in (0,h_0],\ 2\le q\le \infty$,
    \begin{align}\label{e:lbfh}
        \|\chi^w(x,hD)\widecheck\varphi(h^{-1}p^w(x,hD)) f(h)\|_{L^q(\gamma)}\ge \begin{cases}
            C h^{-\rho(q,\sigma)} & \sigma < \infty, \\
            C_\eps h^{-\frac14+\eps} & \sigma = \infty
        \end{cases}
    \end{align}
    with constants $C$, $C_\eps>0$.
\end{prop}

\begin{proof}[Proof that Proposition \ref{prop:sharp} implies the sharpness of Theorem \ref{thm:main}]
We assume $\sigma < \infty$. The other case $\sigma = \infty$ can be dealt with in the same manner.
For $h\in (0, h_0]$, we define a family of functions $g=g(h)$ by
\[
g(h) = \chi^w(x,hD)\widecheck\varphi(h^{-1}p^w(x,hD)) f(h)
\]
where $f(h)$ is a collection of functions as in Proposition \ref{prop:sharp}. Clearly, $g(h)$ satisfies the localization condition \eqref{i:localized}. Since the operators $\chi^w(x,hD), \widecheck\varphi(h^{-1}p^w(x,hD)): L^2(\R^2)\to L^2(\R^2)$ are bounded, $\|g(h)\|_{L^2}\le C$ with a positive constant $C$ independent of $h$. Combining with \eqref{e:lbfh}, we deduce
\begin{equation*}
    \|g(h)\|_{L^q(\gamma)}\ge C h^{-\rho(q, \sigma)}\|g(h)\|_{L^2},\quad h\in(0,h_0]
\end{equation*}
with a constant $C>0$ independent of $h$. Thus, it suffices to show that $g$ is a quasimode, that is, $p^w(x,hD)g = O_{L^2}(h)$. By definition, we write
\begin{align*}
    p^w(x,hD)g &= [p^w(x,hD), \chi^w(x,hD)] \widecheck\varphi(h^{-1}p^w(x,hD)) f\\
    &\qquad\quad + \chi^w(x,hD)p^w(x,hD)\widecheck\varphi(h^{-1}p^w(x,hD)) f,
\end{align*}
where $[ \cdot,\cdot ]$ denotes the commutator. We note that 
\[
[p^w(x,hD), \chi^w(x,hD)] = O_{L^2}(h),\quad h^{-1}p^w(x,hD)\widecheck\varphi(h^{-1}p^w(x,hD)) = O_{L^2}(1).
\]
From this, the desired conclusion follows.
\end{proof}

The proof of Proposition \ref{prop:sharp} hinges on the argument that we have developed when proving the upper bound \eqref{e:main}.
By the discussion at the end of Section \ref{sec:redgloc}, $\gamma$ can be assumed to be a curve parametrized by $\gamma: I\to \R^2$ where $I$ is a small compact interval so that $0\in I$, $\sigma_{\gamma,p}(0) = \sigma$, and $\sigma_{\gamma,p}(s)= 1$ for $s\in I\setminus\{0\}$.

By the $TT^*$ argument, \eqref{e:lbfh} is equivalent to
\begin{align*}
    \|\chi^w(x,hD)\big(\widecheck{\varphi}\big)^2(h^{-1}p^w(x,hD))\chi^w(x,hD)\|_{L^{q'}(\gamma)\to L^q(\gamma)} \ge \begin{cases}
        C h^{-2\rho(q,\sigma)} & \sigma<\infty, \\
        C_\eps h^{-\frac12 + \eps} & \sigma = \infty.
    \end{cases}
\end{align*}
Note that $(\widecheck\varphi)^2 = (\varphi\ast\varphi)^\vee$ and $\varphi\ast\varphi(s) > 0$ for $s$ near $0$. Thus, using \eqref{i:2chivphi} and Proposition \ref{prop:asym1}, one can write
\begin{align}\nonumber
    \chi^w(x,hD)\big(\widecheck{\varphi}\big)^2&(h^{-1}p^w(x,hD))\chi^w(x,hD)(x,y) \\\label{i:lbdec}
    &= \frac{1}{h^2}\int (\varphi\ast\varphi)(s) e^{\frac ih(\phi(s,x,\eta)-\inp{y}{\eta})}b(s,x,\eta;h)d\eta ds + E(x,y)
\end{align}
where $\phi, b$ are as in Proposition \ref{prop:asym1} and $E(x,y)$ is a smoothing error such that the operator $E: f\mapsto \int E(x,y) f(y) dy$ satisfies $E\in O_{\mathcal S'\to \mathcal S}(h^N)$ for a sufficiently large $N>0$.

Now substitute $x=\gamma(u)$, $y=\gamma(v)$, $u,v\in I$. Let
\[
I_1^\circ:= [-\eps_0 \delta, \eps_0 \delta],\quad I_2^\circ:= [(1-\eps_0) \delta, (1+\eps_0) \delta]
\]
for a small constant $\eps_0>0$, where $\delta = h^{1/(2\sigma+1)}$ if $\sigma<\infty$, and $\delta = h^\eps$ with small $\eps>0$ if $\sigma = \infty$. Let $\gamma_1^\circ$, $\gamma_2^\circ$ be images of $I_1^\circ$, $I_2^\circ$ under the map $\gamma$. Now we define the function $f_h^\circ: \gamma\to \R$ so that $f_h^\circ(\gamma(v)) = \chi_{I_2^\circ}(v)$. From now on, we assume that $u\in I_1^\circ$, $v\in I_2^\circ$.

Recall the cutoff function $\psi$ defined in Section \ref{sec:redpropa}. For large $K$, we define $\psi_K$ by
\[
\psi_K(s) := \sum_{k:\,K^{-1}\delta\le 2^{-k}\le K \delta} \psi(2^k s).
\]
We then split the first term in \eqref{i:lbdec} into two terms
\begin{align}\label{t:2terms}
    \frac{1}{h^2}\int (\varphi\ast\varphi)(s)\psi_K(s) e^{\frac ih(\phi(s,\gamma(u),\eta)-\inp{\gamma(v)}{\eta})}b(s,\gamma(u),\eta)d\eta ds + \mathcal E_1(u,v)
\end{align}
where
\begin{align*}
    \mathcal E_1(u,v) := \frac{1}{h^2}\int (\varphi\ast\varphi)(s)\big(1-\psi_K(s)\big) e^{\frac ih(\phi(s,\gamma(u),\eta)-\inp{\gamma(v)}{\eta})}b(s,\gamma(u),\eta)d\eta ds.
\end{align*}
Since $|\gamma(u)-\gamma(v)|\sim \delta$, by \eqref{b:lbphi} and \eqref{r:dseta} the function $(s,\eta)\mapsto \phi(s,\gamma(u),\eta)-\inp{\gamma(v)}{\eta}$ would not have a stationary point on the support of the integrand of $\mathcal E_1$ if we choose $K$ to be sufficiently large. Performing the same calculation as in Section \ref{sec:redpropa}, we deduce that
\[
|\mathcal E_1(u,v)|\lesssim h^{-1} \big(h\delta^{-1}\big)^N,\quad N\gg 1,
\]
implying that the contribution of $\mathcal E_1$ is negligible. 

To estimate the first term in \eqref{t:2terms}, we make use of the stationary phase method as before. Using \cite[Theorem 7.7.5]{Hor83}, we write
\begin{align*}
    &\frac{1}{h^2}\int (\varphi\ast\varphi)(s)\psi_K(s) e^{\frac ih(\phi(s,\gamma(u),\eta)-\inp{\gamma(v)}{\eta})}b(s,\gamma(u),\eta)d\eta ds \\
    &\qquad = h^{-\frac12}\delta^{-\frac12}\mathbf B(u,v) e^{\frac{i}{h}\Phi(u,v)} + \mathcal E_2(u,v)
\end{align*}
where $\Phi$ is as in \eqref{i:defPhi} and
\begin{align*}
    \mathbf B(u,v) := c(\varphi\ast\varphi)&(s(u,v))\psi_K(s(u,v)) \\
    &\times b(s(u,v),\gamma(u),\eta(u,v))\big|\!\det(\delta^{-1}\partial_{(s,\eta)}^2\wt\phi(u,v,0,0))\big|^{-\frac12}
\end{align*}
with $c\in \C\setminus\{0\}$ and $s(u,v), \eta(u,v), \wt\phi(u,v,s,\eta)$ defined in Section \ref{sec:redpropa}.
Also, $\mathcal E_2$ denotes the error term bounded by $h^{\frac12}\delta^{\frac12}$.
By Proposition \ref{prop:asym1}, $|\mathbf B(u,v)|\ge C$ if $(u,v)\in I_1^\circ\times I_2^\circ$ for some constant $C>0$.

Therefore, for $f\in L^{q'}(\gamma)$, $\chi^w(x,hD)\big(\widecheck{\varphi}\big)^2(h^{-1}p^w(x,hD))\chi^w(x,hD)f$ is expressed by
\begin{align}\label{t:Bf}
    h^{-\frac12}\delta^{-\frac12}\int \mathbf B(u,v) e^{\frac{i}{h}\Phi(u,v)} f(v) dv
\end{align}
modulo the aforementioned error terms. Let $\mathbf c_1^\circ, \mathbf c_2^\circ$ be the centers of $I_1^\circ$, $I_2^\circ$, respectively.
Now we define
\[
f_h(v):= \delta^{-\frac{1}{q'}} f_h^\circ(v)e^{-\frac{i}{h}\Phi(\mathbf c_1^\circ, v)}.
\]
Note that $\|f_h\|_{L^{q'}(\gamma)}\lesssim 1$ because $\|f_h^\circ\|_{L^{q'}(\gamma)}\lesssim \delta^{\frac{1}{q'}}$.
By the Taylor theorem and \eqref{e:bdphase}, the function
\begin{align}\label{t:difphase}
    \Phi(u,v) + \Phi(\mathbf c_1^\circ, \mathbf c_2^\circ) - \Phi(u,\mathbf c_2^\circ) - \Phi(\mathbf c_1^\circ,v)
\end{align}
is bounded by $10^{-2}\delta^{2\sigma+1}$ on $I_1^\circ\times I_2^\circ$ provided that $\sigma<\infty$ and $\eps_0$ in the definition of $I_1^\circ$, $I_2^\circ$ is taken to be sufficiently small. If $\sigma = \infty$, one can obtain
\[
\big|\partial_u\partial_v\Phi(u,v)\big|\le C_N \delta(|u|+|u-v|)^N,\quad N\in\N.
\]
This can be shown by using \eqref{ineq:diff4} and the same argument as what we have employed to derive \eqref{i:diff2}. Combining the above estimate with the Taylor theorem implies that \eqref{t:difphase} is bounded by $C_N \delta^N$. We choose $N$ so large that $\eps N > 1$.

Therefore, the absolute value of \eqref{t:Bf} with $f = f_h$ is comparable to $h^{-1/2}\delta^{1/2-1/q'}$ if $u\in I_1^\circ$. 
Clearly $|I_1^\circ|\sim \delta$, thus we get
\begin{align*}
    \big\|\chi^w(x,hD)\big(\widecheck{\varphi}\big)^2(h^{-1}p^w(x,hD))\chi^w(x,hD)f_h\big\|_{L^{q}(\gamma_1^\circ)}\ge \begin{cases}
        C h^{-2\rho(q,\sigma)} & \sigma<\infty, \\
        C_\eps h^{-\frac12 + \eps} & \sigma = \infty,
    \end{cases}
\end{align*}
with constants $C, C_\eps >0$. This proves the desired lower bound.

\subsection{Discussion on the optimality of Corollary \ref{cor:eigencpt}, \ref{cor:eigensch}}
In this subsection, we prove that the estimates \eqref{e:eigencpt}, \eqref{e:eigensch} in Corollary \ref{cor:eigencpt}, \ref{cor:eigensch} are sharp when $f$ is the Laplace-Beltrami eigenfunctions on $\mathbb S^2$ and the Hermite functions on $\R^2$, respectively. By the scaling argument, these sharpness results follow once the following is shown.

\begin{prop}\label{prop:sharpeign}
    Let $M$ be a smooth $2$-dimensional compact Riemannian manifold and $P(h)$ be either $-h^2\Delta_g-1$ on $M$ or $-h^2\Delta + |x|^2 - 1$ on $\R^2$. Suppose that $\gamma: [a,b]\to \R^2$ is a smooth curve in the underlying manifold. Assume further that $\gamma\subset \{|x| < 1\}$ if $P(h) = -h^2\Delta + |x|^2 - 1$ on $\R^2$. Then there exist constants $R\in \N$, $C = C(p, \gamma)>0$ such that there exists a function $n(h): (0,h_0]\to \{k\in \Z: -R\le k\le R\}$ satisfying
    \begin{align}\label{e:lbeign}
        \|\mathds 1_{[n,n+1]}(h^{-1}P(h))\|_{L^2\to L^q(\gamma)}\ge C h^{-\rho(q,\sigma)},\quad 2\le q\le \infty,
    \end{align}
    provided that $\sigma < \infty$. If $\sigma = \infty$, we have \eqref{e:lbeign} with the bound $C h^{-\rho(q,\sigma)}$ replaced by $C_\eps h^{-1/4+\eps}$ for $\eps > 0$.
\end{prop}

\begin{proof}[Proof that Proposition \ref{prop:sharpeign} implies sharpness of Corollary \ref{cor:eigencpt}, \ref{cor:eigensch}]
    We assume $\sigma<\infty$ because the proof for $\sigma = \infty$ follows the same argument.
    We first suppose that $P(h) = -h^2\Delta_g-1$ on $\mathbb S^2$. Note that \eqref{e:lbeign} is equivalent to
    \[
    \|\mathds 1_{[h^{-2}+nh^{-1}, h^{-2}+(n+1)h^{-1}]}(-\Delta_g)\|_{L^2\to L^q(\gamma)}\ge C h^{-\rho(q,\sigma)},\quad 2\le q\le \infty.
    \]
    It is known that the eigenvalues of $-\Delta_g$ are $\{k(k+1): k\in\N_0\}$. From this, it follows that the interval $[h^{-2}+nh^{-1}, h^{-2}+(n+1)h^{-1}]$ in the above contains only one eigenvalue. Thus there exists a sequence $(\lambda_n)_{n=1}^\infty$ of eigenvalues of $\sqrt{-\Delta_g}$ such that
    \[
    \|\mathds 1_{[\lambda_n,\lambda_n+1]}(\sqrt{-\Delta_g})\|_{L^2\to L^q(\gamma)}\ge C \lambda_n^{\rho(q,\sigma)},\quad 2\le q\le \infty,
    \]
    which establishes optimality in the case of $\mathbb S^2$.

    Now we assume $P(h) = -h^2\Delta + |x|^2 - 1$ on $\R^2$. By the standard rescaling $x= h^{1/2}x'$, $\xi= h^{-1/2}\xi'$, \eqref{e:lbeign} is equivalent to
    \[
    \|\mathds 1_{[h^{-1}+n, h^{-1}+n+1]}(H)\|_{L^2\to L^q(\lambda^{1/2}\gamma)}\ge C h^{-\rho(q,\sigma)},\quad 2\le q\le \infty.
    \]
    Since the spectrum of $H$ consists of $2k+2$, $k\in \N_0$, the interval $[h^{-1}+n, h^{-1}+n+1]$ has at most one eigenvalue. Hence the assertion for optimality follows.
\end{proof}

\begin{proof}[Proof of Proposition \ref{prop:sharpeign}]
We only consider the case where $\sigma < \infty$ and $P(h) = -h^2\Delta_g-1$ on $M$. The other cases can be addressed by using the same argument as below. 

Note that Proposition \ref{prop:sharp} gives that, for $2\le q\le \infty$ and $\varphi\in C_c^\infty(\R)$ as in Proposition \ref{prop:sharp},
\begin{align}\label{e:lb2qest}
    \|\chi^w(x,hD)\widecheck\varphi(h^{-1}P(h))\|_{L^2\to L^q(\gamma)}\ge C h^{-\rho(q,\sigma)},\quad h\in (0,h_0].
\end{align}

We now take $\chi(x,\xi) = \chi_{P,\circ}(x)\chi_{F,\circ}(\xi)$, where $\chi_{P,\circ}\in C_c^\infty(M)$ and $\chi_{F,\circ}\in C_c^\infty(\R^2)$ such that $\chi_{F,\circ}\equiv 1$ on $\{\xi: |\xi|\le K\}$ for a sufficiently large $K>0$. By the composition rule, we have $\big(\chi_{P,\circ}\chi_{F,\circ}\big)^w(x,hD) = \chi^w_{P,\circ}(x,hD)\chi^w_{F,\circ}(x,hD) + h r^w(x,hD)$, where $r\in C_c^\infty(T^*M)$. Using the Sobolev inequality, we deduce that the term $h r^w(x,hD)$ on the right-hand side is an error. In fact, making a change of variables, we may assume that $\gamma = \{x_1 = 0\}$. Then by the Sobolev inequality,
\[
\|r^w(x,hD)f\|_{L^q(\gamma)}\le \|r^w(x,hD)f\|_{L^\infty_{x_1}L^q_{x_2}}\lesssim h^{\frac1q-1} \|f\|_{L^2(M)}.
\]
Hence $\|h r^w(x,hD)f\|_{L^2\to L^q(\gamma)}$ is bounded by $C h^{1/q}$, which is much smaller than $h^{-\rho(q,\sigma)}$. 

Note that 
\[
\chi^w_{P,\circ}(x,hD)f = \chi_{P,\circ}(x)f(x),\quad \chi^w_{F,\circ}(x,hD) f(x) = \chi_{F,\circ}(hD)f(x),
\]
thus from \eqref{e:lb2qest}, it follows that
\[
\|\chi_{F,\circ}(hD)\widecheck\varphi(h^{-1}P(h))\|_{L^2\to L^q(\gamma)}\ge C h^{-\rho(q,\sigma)},\quad h\in (0,h_0].
\]
To proceed, for $j\ge \log_2 K$, define $\chi_{F,j}$ to be a dyadic cutoff functions on $\R^2$ such that $\supp(\chi_{F,j})\subset\{2^{j-1}\le |\xi|\le 2^{j+1}\}$ and $\sum_{j\ge \log_2 K} \chi_{F,j}(\xi) = 1$ for $|\xi|\ge K$. Using the triangle inequality, we estimate
\begin{align*}
    \|\chi^w_{F,\circ}(x,hD)&\widecheck\varphi(h^{-1}P(h))\|_{L^2\to L^q(\gamma)} \le \|\widecheck\varphi(h^{-1}P(h)\|_{L^2\to L^q(\gamma)}\\
    &\qquad+ \sum_{j\ge \log_2 K} \big\|\chi^w_{F,j}(x,hD)(1-\chi^w_{F,\circ}(x,hD))\widecheck\varphi(h^{-1}P(h))\big\|_{L^2\to L^q(\gamma)}.
\end{align*}
We prove that the sum on the right-hand side is an error. For that, observe that for $j\ge \log_2 K$, the symbol $|\xi|_g^2-1$ of $-h^2\Delta_g-1$ satisfies the bound $|p(x,\xi)|\gtrsim 2^{2j}$ for $(x,\xi)\in \{2^{j-1}\le |\xi|\le 2^{j+1}\}$ if $K>0$ is sufficiently large. Thus we have
\begin{align*}
    \big\|\chi^w_{F,j}(x,hD)(1-\chi^w_{F,\circ}(x,hD))\widecheck\varphi(h^{-1}P(h)) f\big\|_{L^2}&\lesssim 2^{-2j}\|P(h) \widecheck\varphi(h^{-1}P(h)) f\|_{L^2} \\
    &\lesssim 2^{-2j}h\|f\|_{L^2}.
\end{align*}
Combining with the Sobolev inequality, we get
\begin{align*}
    \big\|\chi^w_{F,j}(x,hD)(1-\chi^w_{F,\circ}(x,hD))\widecheck\varphi(h^{-1}P(h)) f\big\|_{L^q(\gamma)}&\lesssim h2^{-2j} (h^{-1}2^j)^{1-\frac1q}\|f\|_{L^2},
\end{align*}
implying
\begin{align*}
    \sum_{j\ge \log_2 K} \big\|\chi^w_{F,j}(x,hD)(1-\chi^w_{F,\circ}(x,hD))\widecheck\varphi(h^{-1}P(h))\big\|_{L^2\to L^q(\gamma)} \lesssim h^{\frac1q},
\end{align*}
where the resulting bound is sufficiently small compared to $h^{-\rho(q,\sigma)}$.
Consequently, we have
\begin{align}\label{e:lb2qvarp}
    \|\widecheck\varphi(h^{-1}P(h))\|_{L^2\to L^q(\gamma)}\ge C h^{-\rho(q,\sigma)},\quad h\in (0,h_0].
\end{align}
The smoothed projection operator $\widecheck\varphi(h^{-1}P(h))$ in the above can be replaced by $\mathds 1_{[n,n+1]}(h^{-1}P(h))$ for some $n\in\N$. Indeed, we can prove that for $h\in (0,h_0]$, there exists $n = n(h)\in \N$ such that $-R\le n\le R$ for $R>0$ independent of $h$ and
\begin{align}\label{e:lb2qproj}
    \|\mathds 1_{[n,n+1]}(h^{-1}P(h))\|_{L^2\to L^q(\gamma)}\ge C h^{-\rho(q,\sigma)},\quad 2\le q\le \infty,
\end{align}
which completes the proof of the proposition. The implication from \eqref{e:lb2qvarp} to \eqref{e:lb2qproj} can be shown by following an argument in \cite[Section 5.2]{BGT04}, but we provide its proof here for completeness.

We write
\[
\widecheck\varphi(h^{-1}P(h)) = \sum_{n\in \Z} \mathds 1_{[n,n+1]}(h^{-1}P(h))\mathds 1_{[n,n+1]}(h^{-1}P(h))\widecheck\varphi(h^{-1}P(h)).
\]
Note that
\[
\big\|\mathds 1_{[n,n+1]}(h^{-1}P(h))\widecheck\varphi(h^{-1}P(h))\big\|_{L^2\to L^2}\lesssim (1+|n|)^{-N},\quad n\in \Z
\]
for a large natural number $N$. Combining with the triangle inequality and \eqref{e:lb2qvarp}, we obtain
\[
\sum_{n\in \Z}\,(1+|n|)^{-N} \big\|\mathds 1_{[n,n+1]}(h^{-1}P(h))\big\|_{L^2\to L^q(\gamma)}\ge C h^{-\rho(q,\sigma)},\quad 2\le q\le \infty.
\]
From this and the upper bound \eqref{e:eigencpt}, it follows that there exists a constant $R>0$ independent of $h$ such that
\[
\sum_{|n|\le R}\, \big\|\mathds 1_{[n,n+1]}(h^{-1}P(h))\big\|_{L^2\to L^q(\gamma)}\ge C h^{-\rho(q,\sigma)},\quad 2\le q\le \infty.
\]
The desired assertion \eqref{e:lb2qproj} now follows from the Pigeonhole principle.
\end{proof}

\end{document}